\documentclass[11pt]{article}
\date{}

\usepackage{tikz}
\usetikzlibrary{calc}
\usetikzlibrary{matrix}
\usepackage{amsmath,amssymb,amsfonts,amsthm,subfig}
\usepackage[pdftex]{hyperref} 

\allowdisplaybreaks

\usepackage{graphicx}

\newcommand{\V}[1]{\mbox{\boldmath $ #1 $}}

\def \p{\partial}

\newcommand{\bey}{\begin{eqnarray}}
\newcommand{\eey}{\end{eqnarray}}
\newcommand{\nn}{\nonumber}
\newcommand{\beq}{\begin{equation}}
\newcommand{\eeq}{\end{equation}}
\theoremstyle{plain}
\newtheorem{thm}{\hspace{6mm}Theorem}[section]

\theoremstyle{definition}

\theoremstyle{remark}
\newtheorem{exam}{\hspace{6mm}Example}[section]
\newtheorem{rem}{\hspace{6mm}Remark}[section]

\title{Unconditionally stable high-order time integration for
moving mesh finite difference solution of linear convection-diffusion equations}

\author{Weizhang~Huang
\thanks{Department of Mathematics, 
the University of Kansas, Lawrence, KS 66045, U.S.A. 
({\em huang@math.ku.edu}). This work was supported
in part by the NSF under Grants DMS-0712935 and DMS-1115118.}
}

\begin{document}
\vskip 1cm
\maketitle

\begin{abstract}
This paper is concerned with moving mesh finite difference solution of partial differential equations.
It is known that mesh movement introduces an extra convection term and its numerical treatment has
a significant impact on the stability of numerical schemes. Moreover, many implicit second and
higher order schemes, such as the Crank-Nicolson scheme, will loss their unconditional stability.
A strategy is presented for developing temporally high order, unconditionally stable finite difference schemes
for solving linear convection-diffusion equations using moving meshes. 
Numerical results are given to demonstrate the theoretical findings.
\end{abstract}



\noindent{\bf Abbreviated title.}
Unconditionally stable high order time integration on moving meshes.

\section{Introduction}

In the last two decades, moving mesh methods have attracted considerable attention
from scientists and engineers and been successfully applied to
a variety of problems; e.g., see \cite{BHR09,HR11} and references therein.
A common feature of those methods is that they employ a time varying
mesh to follow the motion of the physical boundary and/or certain solution properties.
Generally speaking, partial differential equations (PDEs) can be discretized
directly on a moving mesh using finite element and finite volume methods.
When a finite difference method is desired, it is common practice to transform
PDEs to a reference domain and discretize them on a fixed mesh defined thereon.
The interested reader is referred to \cite{HR11} for more detailed discussion
on the discretization of PDEs on moving meshes. 
It should be emphasized that, despite what discretization method is used, the movement of the mesh
inevitably introduces an extra convection term into PDEs, whose numerical treatment often
has a significant impact on the stability and convergence of moving mesh methods. The term also
makes the analysis of stability and convergence much more difficult for moving mesh methods
even when linear PDEs are involved.

While a number of convergence results have been developed for the numerical solution of
two-point boundary value problems using equidistributing meshes (a type of moving meshes)
(e.g., see \cite{BM00,BM01,BM01b,CX08,HH08a,Hua05d,KS01,Mac99,QS99,QST00}),
theoretical studies of moving mesh methods for time dependent PDEs are far from complete.
For example, for linear convection-diffusion problems 
Dupont \cite{Dup82}, Bank and Santos \cite{BS93}, Dupont and Liu \cite{DL02},
and Liu et al. \cite{LBDGS03} establish symmetric error estimates for various finite element methods
(FEM), including semi-discrete FEM,
semi-discrete mixed FEM, FEM-implicit Euler, and space-time FEM.
These results essentially require the conditions (or their discrete counterparts)
\bey
&& | \dot{\V{x}}(\V{x},t) - \V{b}(\V{x},t)| \le C_1,\quad \forall \V{x} \in \Omega, \;  t > 0
\label{BS93-1}
\\
&& | \nabla \cdot \dot{\V{x}}(\V{x},t)| \le C_2 ,\quad \forall \V{x} \in \Omega, \;  t > 0
\label{BS93-2}
\eey
where
$C_1$ and $C_2$ are positive constants, $\Omega$ is the physical domain,
$\dot{\V{x}}$ is the mesh speed, and $\V{b}$ is the coefficient of the convection term (see (\ref{ibvp}) below).
Condition (\ref{BS93-1}) requires that the mesh move not too fast with reference to $\V{b}$.
To see the geometric meaning of (\ref{BS93-2}), we view the moving mesh as the image of
a computational mesh under a coordinate transformation
$\V{x} = \V{x}(\V{\xi}, t)$: $\Omega_c \to \Omega$, where $\Omega_c$
is the computational domain. Denote the Jacobian of the coordinate transformation by $J$,
i.e., $J = \mbox{det}(\p \V{x}/\p \V{\xi})$. It is known \cite{HR11} that
\[
\frac{1}{J} \dot{J} = \nabla \cdot \dot{\V{x}} .
\]
Then (\ref{BS93-2}) reduces to
\beq
|\dot{J}| \le C_2 |J| .
\label{BS93-3}
\eeq
Since $J$ represents the volume of mesh elements, (\ref{BS93-3}), or equivalently (\ref{BS93-2}),
requires that the relative change of the volume of mesh elements be in a constant order.

Formaggia and Nobile \cite{FN99,FN04} and Boffi and Gastaldi \cite{BG04} study the relation between
stability and satisfaction of the geometric conservation law (GCL) \cite{TL79} in the Arbitrary-Langragian-Eulerian
formulation \cite{HAC74} with finite element spatial discretization. They show that satisfying GCL
is neither a necessary nor a sufficient condition for the stability of a scheme although it often helps improve
accuracy and enhance stability. In particular, Boffi and Gastaldi show that the FEM-implicit Euler scheme
is only conditionally stable.
Formaggia and Nobile \cite{FN99,FN04} present several modifications of the implicit Euler, Crank-Nicolson,
and BDF2 (two step backward differentiation formula \cite{CH52,Hen62}) schemes,
which can be made to satisfy GCL. They show that the FEM-modified implicit Euler scheme
is unconditionally stable whereas the FEM-modified Crank-Nicolson and FEM-BDF(2) schemes
are only conditionally stable with the maximum allowable time step depending on $\nabla \cdot \dot{\V{x}}$.

Ferreira \cite{Fer97} studies a finite difference (FDM)-implicit Euler scheme applied to the nonconservative form
of transformed linear convection-diffusion equations and shows that the scheme is stable
and convergent under discrete analogs of conditions (\ref{BS93-1}) and (\ref{BS93-3}).
Mackenzie and Mekwi \cite{MM07} consider an FDM-$\theta$ scheme
for the conservative form of transformed linear convection-diffusion equations. They show that
the FDM-implict Euler scheme is unconditionally stable in an energy norm but the Crank-Nicolson scheme
is only conditionally stable. They also show that the FDM-$\theta$ scheme
can be made to be unconditionally stable and of second order in time (and in space)
when $\theta$ is properly chosen depending on the mesh. Although this variable $\theta$ scheme is
unconventional, there are no other moving mesh methods that are known to be of second
or higher order in time and unconditionally stable. 
More recently, Huang and Schaeffer \cite{HS11} show that several FDM-$\theta$ moving mesh schemes
for one dimensional convection-diffusion problems are stable in $L^\infty$ norm 
under the conventional CFL condition and a mesh speed related condition which can be roughly expressed as
\beq
| \dot{x} - b | \le \frac{C_3}{h},
\label{HS-1}
\eeq
where $C_3$ is a constant and $h$ is the maximum spacing of the mesh. Notice that this condition is weaker
than (\ref{BS93-1}).

It is well known that high-order methods are practically important in enhancing the computational accuracy
and efficiency. It is also important in theory to know if unconditionally stable high-order methods exist
in the moving mesh context. The objective of this paper is to develop unconditionally stable,
high-order (in time) moving finite difference schemes for the solution of
the initial-boundary value problem (IBVP) of a general linear convection-diffusion equation,
\beq
\begin{cases}
\frac{\p u}{\p t} + \nabla \cdot (u \V{b}) + c u = \nabla \cdot (a \nabla u) + f ,\quad \forall \V{x} \in \Omega, \; t \in (0,T]
\\
u(\V{x},t) = g(\V{x},t),\quad \forall \V{x} \in \p \Omega
\\
u(\V{x}, 0) = u^0(\V{x}), \quad \forall \V{x} \in \Omega
\end{cases}
\label{ibvp}
\eeq
where $\Omega \subset \mathbb{R}^d$ ($d = 1,2$) is the physical domain with probably moving
or deforming boundary and $a = a(\V{x},t)$, 
$\V{b} = \V{b}(\V{x},t)$, $c=c(\V{x}, t)$, $f = f(\V{x}, t)$, $g = g(\V{x}, t)$, and $u^0(\V{x})$ are given,
sufficiently smooth functions. For the posedness of IBVP, we assume that the coefficients satisfy
\bey
&& 0 < \underline{a} \le a(\V{x},t) \le \overline{a} < \infty, \quad \forall \V{x} \in \Omega,\; t \in (0,T]
\label{coef-1}
\\
&& c(\V{x}, t) + \frac{1}{2} \nabla \cdot \V{b}(\V{x}, t) \ge 0, \quad \forall \V{x} \in \Omega,\; t \in (0,T]
\label{coef-2}
\eey
where $\underline{a}$ and $\overline{a}$ are two positive constants. 
We note that for stability analysis, we only need to consider the homogeneous system,
i.e., the IBVP with $f \equiv 0$ and $g \equiv 0$. In this situation, it is easy to show that
the solution satisfies the stability inequality
\beq
\frac{d }{d t} \int_\Omega u^2 d \V{x}  \le 0.
\label{si-1}
\eeq
The key to our development is the preservation of stability inequality (\ref{si-1}).
We use the method of lines approach and first discretize the PDE in space in such a way
that the resulting system of ordinary differential equations (ODEs) satisfies a semi-discrete
stability inequality. Then high-order schemes satisfying a fully discrete version of (\ref{si-1})
are constructed for integrating the ODE system. The resulting schemes are stable
in the energy norm corresponding to the fully discrete stability inequality.

An outline of this paper is given as follows. In \S 2, high-order unconditionally stable schemes for
ODE systems satisfying a stability inequality are developed. Strategies for spatially
discretizing IBVP (\ref{ibvp}) into such an ODE system in one and two dimensions are explored
in \S 3 and \S 4. Numerical examples obtained with the developed schemes
are presented in \S 5. Finally, \S 6 contains conclusions and comments. 

\section{High-order unconditionally stable schemes for nonautonomous ODE systems}

In this section we present an approach of constructing unconditionally stable schemes for
the initial value problem of a nonautonomous ODE system
\beq
\begin{cases}
M(t) \frac{d \V{u}}{d t} = A(t) \V{u} + \V{f}(t),\\
\V{u}(0) = \V{u}^0
\end{cases}
\label{ode-3}
\eeq
where $M(t)$ and $A(t)$ are $l \times l$ matrices (for some positive integer $l$) and $\V{f}(t)$ is a given vector-valued
function. We assume that  $M(t)$ is symmetric and positive definite and $A(t)+ \sqrt{M}(t)\frac{d \sqrt{M}}{d t}(t)$ is
negative semi-definite, i.e., 
\beq
\V{z}^T \left (A(t)+ \sqrt{M}(t)\frac{d \sqrt{M}}{d t}(t)\right ) \V{z} \le 0,\quad \forall \V{z} \in \mathbb{R}^l,\quad \forall t > 0
\label{cond-1}
\eeq
where $\sqrt{M}$ denotes the square root of $M$. 
As will be seen in \S\ref{SEC:1d} and \S\ref{SEC:2d}, such systems arise from
finite difference discretization of IBVP (\ref{ibvp}) on moving meshes, with
$M(t)$ and $A(t)$ being the mass and stiffness matrices, respectively.

The idea of the approach is simple. Indeed, we rewrite (\ref{ode-3}) into
\beq
 \frac{d }{d t}\left ( \sqrt{M} \V{u} \right ) = (\sqrt{M})^{-1}\left (A +  \sqrt{M}\frac{d \sqrt{M}}{d t}\right )\V{u} 
 + (\sqrt{M})^{-1} \V{f},
\label{ode-5}
\eeq
which can further be written as
\beq
\frac{d \V{v}}{d t} = B(t) \V{v} + (\sqrt{M}(t))^{-1} \V{f}(t) ,
\label{ode-v-1}
\eeq
where 
\bey
&& \V{v}(t) = \sqrt{M}(t) \V{u}(t) ,
\label{v-1}
\\
&& B(t) = (\sqrt{M})^{-1} \left (A +  \sqrt{M}\frac{d \sqrt{M}}{d t}\right )(\sqrt{M})^{-1}.
\label{B-1}
\eey
Then, unconditionally stable schemes can be obtained by applying conventional implicit schemes to (\ref{ode-v-1}).
The initial condition for the new variable is
\beq
\V{v}(0) = \V{v}^0 \equiv \sqrt{M}(0) \V{u}^0.
\label{ic-v-1}
\eeq
Moreover, assumption (\ref{cond-1}) implies that $B(t)$ is negative semi-definite for any $t > 0$.

In the following, we illustrate the idea with three schemes. To this end, we assume that
a partition is given for $[0, T]$,
\[
t_0 = 0 < t_1 < \cdots < t_N = T .
\]
We denote $\Delta t_n = t_{n+1} - t_n$.

We first apply the backward Euler discretization to (\ref{ode-v-1}). It gives
\beq
  \frac{\V{v}^{n+1}-\V{v}^n}{\Delta t_n} = B(t_{n+1}) \; \V{v}^{n+1}
  + (\sqrt{M}(t_{n+1}))^{-1} \V{f}(t_{n+1})  .
\label{ImEuler}
\eeq
Since $B$ is negative semi-definite, for the homogeneous situation (with $\V{f} = 0$)
we can readily show that (\ref{ImEuler}) satisfies
\beq
(\V{v}^{n+1})^T\V{v}^{n+1} \le (\V{v}^{n})^T \V{v}^{n},
\label{si-3}
\eeq
which can be written in terms of the original variable $\V{u}$ as
\beq
(\V{u}^{n+1})^T \sqrt{M}(t_{n+1}) \V{u}^{n+1} \le (\V{u}^{n})^T \sqrt{M}(t_{n}) \V{u}^{n} .
\label{si-4}
\eeq
Inequality (\ref{si-4}) is a discrete analog to (\ref{si-1}). It implies that the backward Euler scheme (\ref{ImEuler}),
which is of first order, is unconditionally stable.

A second order unconditionally stable scheme can be obtained by applying the midpoint discretization
to (\ref{ode-v-1}). Indeed, we have 
\beq
 \frac{\V{v}^{n+1}-\V{v}^n}{\Delta t_n} = B\left (\frac{t_{n}+t_{n+1}}{2}\right ) \frac{\V{v}^{n+1}+\V{v}^n}{2}
 + \left (\sqrt{M}\left (\frac{t_{n}+t_{n+1}}{2}\right )\right )^{-1} \V{f}\left (\frac{t_{n}+t_{n+1}}{2}\right ) .
\label{siip2-1}
\eeq
It is easy to show that the solution of (\ref{siip2-1}) (with $\V{f} = 0$) satisfies stability inequality (\ref{si-3}).

Higher order unconditionally stable schemes can be obtained by applying collocation schemes to (\ref{ode-v-1}).
To explain this, we consider the $m$ Gauss-Legendre points in $(0,1)$, $\rho_1$, ..., $\rho_m$, and
denote the collocation points by
\beq
t_{n,j} = t_n + \rho_j \Delta t_{n} , \quad j = 1, ..., m .
\label{col-points}
\eeq
Applying the $m$-point collocation scheme (e.g., see \cite{AMR88,Asc08}) to (\ref{ode-v-1}), we get,
for $n = 0, ..., N-1$,
\beq
 \frac{d \V{v}_h}{d t}(t_{n,j}) = B (t_{n,j}) \V{v}_h(t_{n,j}) + (\sqrt{M}(t_{n,j}))^{-1} \V{f}(t_{n,j}) ,\quad  j = 1, ..., m
\label{col-1}
\eeq 
where $\V{v}_h(t)$, an approximation to the exact solution $\V{v}(t)$, is continuous on $[0, T]$ and a polynomial
of degree $m$ on each subinterval $[t_n, t_{n+1}]$.

We can rewrite (\ref{col-1}) into a more explicit form. Let $\tilde{\rho}_1$, ..., $\tilde{\rho}_{m-1}$ be
the $(m-1)$ Gauss-Legendre points in $(0,1)$ and denote $\tilde{\rho}_0 = 0$ and $\tilde{\rho}_m = 1$. Define
\beq
\tilde{t}_{n,j} = t_n + \tilde{\rho}_j \Delta t_n, \quad j = 0, ..., m.
\label{interp-points}
\eeq
(We may also take $\tilde{\rho}_0,...,\tilde{\rho}_m$ as the $(m+1)$ Gauss-Lobatto Legendre points on $[0,1]$.)
Then, $\V{v}_h$ can be expressed as
\beq
\V{v}_h(t) = \sum_{k=0}^m \tilde{\V{v}}_{n,k} \tilde{\ell}_k\left (\frac{t-t_n}{\Delta t_n}\right),\quad \forall t \in [t_n, t_{n+1}],\;
n = 0, ..., N-1
\label{v-2}
\eeq
where $\tilde{\V{v}}_{n,k} \approx \V{v}(\tilde{t}_{n,k})$ ($k = 0, ..., m$) are the unknown variables
and $\tilde{\ell}_k$'s are the Lagrange polynomials associated with nodes $\tilde{\rho}_0$, ..., $\tilde{\rho}_{m}$, i.e.,
\[
\tilde{\ell}_k(\rho) = \prod\limits_{i= 0\atop i \neq k}^m \frac{\rho - \tilde{\rho}_i}{\tilde{\rho}_k-\tilde{\rho}_i} .
\]
Inserting (\ref{v-2}) into (\ref{col-1}), we have
\beq
 \frac{1}{\Delta t_n} \sum\limits_{k=0}^m \tilde{\V{v}}_{n,k} \tilde{\ell}_k^{\; '}\left (\rho_j\right )
 = B (t_{n,j}) \sum\limits_{k=0}^m \tilde{\V{v}}_{n,k} \tilde{\ell}_k\left (\rho_j\right )
 + (\sqrt{M}(t_{n,j}))^{-1} \V{f}(t_{n,j}),
\quad  j = 1, ..., m
\label{col-2}
\eeq 
This equation, together with the initial condition
\beq
\tilde{\V{v}}_{n,0}  = \tilde{\V{v}}_{n-1,m}, 
\label{v-4}
\eeq
forms a system of ODEs for unknowns $\tilde{\V{v}}_{n,k}$, $k = 0, ..., m$.

Once $\tilde{\V{v}}_{n,m}$ has been obtained, we can compute
the original variable $\V{u}$ at $t = t_{n+1}$ by
\beq
\V{u}^{n+1} = (\sqrt{M})^{-1}(t_{n+1}) \; \tilde{\V{v}}_{n,m}, \quad n=0, ..., N-1.
\label{v-5}
\eeq
Moreover, it is known (e.g., see \cite{AMR88,Asc08}) that the convergence order of the $m$-point collocation scheme
(\ref{col-1}) is $2 m$.
(It is easy to verify that scheme (\ref{col-1}) reduces to the midpoint scheme (\ref{siip2-1}) when $m=1$.)

In the following we show that scheme (\ref{col-1}) (with $\V{f} = 0$) satisfies (\ref{si-3}).
To this end, we denote $\V{v}^{n+1} = \tilde{\V{v}}_{n,m}$ for $n = 0, ..., N-1$.
Obviously, $\V{v}^{n+1} = \V{v}_h(t_{n+1})$.
Let $\omega_1,\, \cdots,\, \omega_m$ be the weights for the Gaussian quadrature rule associated with the points
$\rho_1, \, \cdots,\, \rho_{m}$, viz., 
\[
\int_0^1 f(\rho ) d \rho \approx \sum_{j=1}^m \omega_j f(\rho_j) .
\]
It is known that $\omega_j$'s are positive
and the quadrature rule is exact for all polynomials of degree up to $(2m-1)$. 
Multiplying (\ref{col-1}) with $\Delta t_n \omega_j \V{v}_h(t_{n,j})^T $ and
summing the result over $j$, we have
\[
\Delta t_n \sum_{j=1}^m \omega_j \V{v}_h(t_{n,j})^T \frac{d \V{v}_h}{d t} (t_{n,j})
= \Delta t_n \sum_{j=1}^m \omega_j \V{v}_h(t_{n,j})^T B(t_{n,j}) \V{v}_h(t_{n,j}).
\]
The negative semi-definiteness of $B$ implies that the right-hand side is nonpositive.
Moreover,  $\V{v}_h^T \frac{d \V{v}_h}{d t}$ is a polynomial of degree at most $(2 m-1)$ and
the sum on the left-hand side is equal to the integral of $\V{v}_h^T \frac{d \V{v}_h}{d t}$ over $[t_n, t_{n+1}]$.
Thus, we have
\[
\int_{t_n}^{t_{n+1}} \V{v}_h^T \frac{d \V{v}_h}{d t} d t \le 0.
\]
This, combined with the fact that
\[
\int_{t_n}^{t_{n+1}} \V{v}_h^T \frac{d \V{v}_h}{d t} d t
= \frac{1}{2} \int_{t_n}^{t_{n+1}} \frac{d}{d t} \left (\V{v}_h^T \V{v}_h\right ) d t
= \frac{1}{2} ( (\V{v}^{n+1})^T \V{v}^{n+1} - (\V{v}^{n})^T \V{v}^{n} ) ,
\]
leads to (\ref{si-3}).

\vspace{10pt}

It is remarked that scheme (\ref{col-1}) is expressed in the new variable $\V{v}$.
It can be reformulated in terms of the original variable $\V{u}$ using the relation (\ref{v-1}).
Moreover, scheme (\ref{col-1}) involves $\sqrt{M}$ and its inverse whose computation can be expensive
in general. In our current situation with finite difference discretization for PDEs,
however, the mass matrix $M$ is diagonal and its square root and the inverse can be
computed easily (see the next two sections) and thus the involvement of $\sqrt{M}$
and its inverse has very mild effects on the computational efficiency.
The argument also goes for finite volume methods or finite element methods with lumped mass matrix. 

\section{1D convection-diffusion equations}
\label{SEC:1d}

In this and next sections, we study the finite difference discretization of IBVP (\ref{ibvp}).
Our goal is to obtain schemes that satisfy the property (\ref{cond-1}) and thus unconditionally stable
integration schemes can be developed.
We consider one dimensional convection-diffusion equations in this section and two dimensional ones in
the next section.
 
In one dimension, IBVP (\ref{ibvp}) becomes
\beq
\begin{cases}
u_t + (b u)_x + c u = (a u_x)_x + f,\qquad \forall x\in (x_l(t), x_r(t)), \quad 0 < t \le T
\\
u(x_l(t), t) = g(x_l(t), t),\quad u(x_r(t), t) = g(x_r(t), t),
\\
u(x, t) = u^0(x).
\end{cases}
\label{ibvp-1d}
\eeq
Notice that this setting permits moving domains.

For spatial discretization, we assume that a moving mesh is given at time instants $t_0,\, \cdots,\, t_N$, i.e.,
\[
x_0^n = x_l(t_n) < x_1^n < \cdots < x_{J_{max}}^n = x_r(t_n) ,\quad n = 0, \cdots, N.
\]
The mesh points are considered to vary linearly on each tim subinterval, viz.,
\[
x_j(t) = \frac{t_{n+1}-t}{\Delta t_n} x_j^n + \frac{t-t_n}{\Delta t_n} x_j^{n+1}, \quad
\forall t \in [t_n, t_{n+1}],\; j = 0, ..., J_{max}.
\]
Notice that the mesh speed on $[t_n, t_{n+1}]$ is constant, viz.,
\[
\dot{x}_j(t) =  \frac{x_j^{n+1}-x_j^n}{\Delta t_n} , \quad \forall t \in [t_n, t_{n+1}],\; j = 0, ..., J_{max}.
\]
We view this moving mesh as the image of a uniform computational
mesh under a coordinate transformation. Denote the coordinate transformation by
$x=x(\xi, t)$: $[0,{J_{max}}] \to [x_l(t),\; x_r(t)]$. Then the moving mesh can be expressed as
\[
x_j(t) = x(\xi_j, t), \quad \xi_j=j,\quad  j = 0, ..., {J_{max}}.
\]

To discretize the PDE (\ref{ibvp-1d}) on the moving mesh, we first transform it from the physical domain
into the computational domain using the coordinate transformation. It is easy to show (e.g., see \cite{HR11})
that the transformed PDE can be written either in a conservative form as
\beq
x_\xi \dot{u} + \dot{x}_\xi  u + \frac{\p}{\p \xi} ((b-\dot{x}) u) + x_\xi c u = \frac{\p }{\p \xi}
\left (\frac{a}{x_\xi} \frac{\p  u}{\p \xi}\right ) + x_\xi f ,
\label{pde-3}
\eeq
or in a non-conservative form as
\beq
x_\xi \dot{u} - \dot{x} \frac{\p u}{\p \xi} +  \frac{\p}{\p \xi} ( bu)
+ x_\xi c u = \frac{\p }{\p \xi} \left (\frac{a}{x_\xi} \frac{\p  u}{\p \xi}\right ) + x_\xi f,
\label{pde-4}
\eeq
where $\dot{x}$ (the mesh speed) and $\dot{u}$ denote the time derivatives of $x$ and $u$ in the new variables
$t$ and $\xi$, respectively.

We now consider a central finite difference scheme based on the conservative form (\ref{pde-3}). It reads as
\bey
&& \frac{h_{j+1}+h_j}{2} \dot{u}_j + \frac{\dot{h}_{j+1}+\dot{h}_j}{2} u_j
+ (b_{j+\frac 12}-\dot{x}_{j+\frac 12})\frac{u_{j+1}+u_j}{2} 
- (b_{j-\frac 12}-\dot{x}_{j-\frac 12})\frac{u_{j}+u_{j-1}}{2}  \nn \\
&&\qquad \qquad  + \; \frac{h_{j+1}+h_j}{2} c_j u_j = a_{j+\frac 12}\frac{u_{j+1}-u_{j}}{h_{j+1}}
- a_{j-\frac 12}\frac{u_{j}-u_{j-1}}{h_{j}} + \frac{h_{j+1}+h_j}{2} f_j,
\label{fd-2}
\eey
where $u_j = u_j(t) \approx u(x_j(t), t)$, $\dot{x}_{j+\frac 1 2} = (\dot{x}_j + \dot{x}_{j+1})/2$,
and $h_j = h_j(t) = x_{j}(t) - x_{j-1}(t)$. 
The boundary condition has the discrete form as
\beq
u_0(t) = g(x_l(t), t), \quad u_{J_{max}}(t) = g(x_r(t), t).
\label{bc-3}
\eeq
Notice that a special treatment has
been used in (\ref{fd-2}) for the convection term, viz.,
\bey
\left. \frac{\p}{\p \xi} ((b-\dot{x}) u)\right |_{x_j} & \approx &
(b-\dot{x})_{j+\frac 1 2} u_{j+\frac 1 2} - (b-\dot{x})_{j-\frac 1 2} u_{j-\frac 1 2}
\nn \\
& \approx &
(b-\dot{x})_{j+\frac 1 2} \frac{u_{j} + u_{j+1}}{2} - (b-\dot{x})_{j-\frac 1 2} \frac{u_{j} + u_{j-1}}{2} .
\nn
\eey
In words, the flux $(b-\dot{x}) u$ is approximated at half mesh points $x_{j+\frac 12}$.
This treatment is crucial for scheme (\ref{fd-2}) to satisfy condition (\ref{cond-1}).

Scheme (\ref{fd-2}) can be cast into the matrix form (\ref{ode-3}), with the mass and stiffness matrices given by,
for $j = 1, ..., J_{max}-1$,
\bey
(M \V{u})_j  & = & \frac{h_{j+1}+h_j}{2} u_j,
\label{mass-1}
\\
(A\V{u})_j & = & a_{j+\frac 12}\frac{u_{j+1}-u_{j}}{h_{j+1}} - a_{j-\frac 12}\frac{u_{j}-u_{j-1}}{h_{j}} 
 -  \frac{\dot{h}_{j+1}+\dot{h}_j}{2} u_j  - \frac{h_{j+1}+h_j}{2} c_j u_j  \nn \\
&& - \;  (b_{j+\frac 12}-\dot{x}_{j+\frac 12})\frac{u_{j+1}+u_j}{2} 
+ (b_{j-\frac 12}-\dot{x}_{j-\frac 12})\frac{u_{j}+u_{j-1}}{2} ,
\label{stiff-1}
\eey
where $\V{u} = [u_0, ..., u_{J_{max}}]^T$.
Notice that the mass matrix $M$ is diagonal and its square root is
\beq
\sqrt{M} = \mbox{diag}\left (\sqrt{\frac{h_{j+1}+h_j}{2}} \right ) .
\label{mass-2}
\eeq
Moreover, 
\beq
\frac{d }{d t} \sqrt{\frac{h_{j+1}+h_j}{2}} = \frac{\sqrt{2}}{4} \frac{(\dot{h}_{j+1}+\dot{h}_j)}{\sqrt{h_{j+1}+h_j}} .
\label{mass-3}
\eeq

\begin{thm}
\label{thm3.1}
Assume that there holds
\beq
c_j + \frac{1}{2} \frac{(b_{j+\frac 12} - b_{j-\frac 12})}{(h_{j+1}+h_j)/2} \ge 0,
\quad j = 1, ..., {J_{max}}-1.
\label{coef-3}
\eeq
Then, the finite difference scheme (\ref{fd-2}) satisfies the condition (\ref{cond-1}).
As a consequence,
the fully discrete scheme resulting from the application of the time integration (\ref{col-1}) to (\ref{ode-3})
with $M$ and $A$ defined in (\ref{mass-1}) and (\ref{stiff-1})
is unconditionally stable and of order $(2m)$ in time and order 2 in space.
\end{thm}

\begin{proof}
Recall that we only need to consider the homogeneous situation with $f \equiv 0$ and $g \equiv 0$.
From (\ref{bc-3}), (\ref{stiff-1}), (\ref{mass-2}), and (\ref{mass-3}), we have
\bey
&& \V{u}^T (A + \sqrt{M} \frac{d }{d t} \sqrt{M}) \V{u} 
\nn \\
& = & \sum_{j=1}^{{J_{max}}-1} u_j \left [a_{j+\frac 12}\frac{u_{j+1}-u_{j}}{h_{j+1}} - a_{j-\frac 12}\frac{u_{j}-u_{j-1}}{h_{j}}\right ]
\nn\\
&& + \sum_{j=1}^{{J_{max}}-1} u_j \left [- (b_{j+\frac 12}-\dot{x}_{j+\frac 12})\frac{u_{j+1}+u_j}{2} 
+ (b_{j-\frac 12}-\dot{x}_{j-\frac 12})\frac{u_{j}+u_{j-1}}{2} \right ]
\nn \\
& & - \sum_{j=1}^{{J_{max}}-1} \frac{\dot{h}_{j+1}+\dot{h}_j}{2} u_j^2 - \sum_{j=1}^{{J_{max}}-1} \frac{h_{j+1}+h_j}{2} c_j u_j^2
 + \sum_{j=1}^{{J_{max}}-1} \frac{\dot{h}_{j+1}+\dot{h}_j}{4 } u_j^2 .
\label{e3-1}
\eey
For the first term on the right-hand side, using the boundary condition (\ref{bc-3}) and summation by parts we get
\bey
&& \sum_{j=1}^{{J_{max}}-1} u_j \left [a_{j+\frac 12}\frac{u_{j+1}-u_{j}}{h_{j+1}} - a_{j-\frac 12}\frac{u_{j}-u_{j-1}}{h_{j}}\right ]
\nn \\
& = & \sum_{j=2}^{{J_{max}}} u_{j-1} a_{j-\frac 12}\frac{u_{j}-u_{j-1}}{h_{j}} 
- \sum_{j=1}^{{J_{max}}-1} u_j a_{j-\frac 12}\frac{u_{j}-u_{j-1}}{h_{j}}
\nn \\
& = & - \sum_{j=1}^{{J_{max}}} a_{j-\frac 12}\frac{(u_{j}-u_{j-1})^2}{h_{j}} .
\label{e3-2}
\eey
For the second term, we have
\bey
&& \sum_{j=1}^{{J_{max}}-1} u_j \left [- b_{j+\frac 12}\frac{u_{j+1}+u_j}{2} 
+ b_{j-\frac 12}\frac{u_{j}+u_{j-1}}{2} \right ]
\nn \\
& = & - \sum_{j=2}^{{J_{max}}} u_{j-1} b_{j-\frac 12}\frac{u_{j}+u_{j-1}}{2} 
+ \sum_{j=1}^{{J_{max}}-1} u_j b_{j-\frac 12}\frac{u_{j}+u_{j-1}}{2}
\nn \\
& = & \sum_{j=2}^{{J_{max}}-1} b_{j-\frac 12}\frac{u_{j}^2-u_{j-1}^2}{2} 
- u_{{J_{max}}-1} b_{{J_{max}}-\frac 12}\frac{u_{{J_{max}}}+u_{{J_{max}}-1}}{2}  + u_1 b_{1-\frac 12}\frac{u_{1}+u_{0}}{2}
\nn \\
& = & \sum_{j=1}^{{J_{max}}} b_{j-\frac 12}\frac{u_{j}^2-u_{j-1}^2}{2} 
\nn \\
& = & \frac{1}{2} \sum_{j=1}^{{J_{max}}} b_{j-\frac 12}u_{j}^2 - \frac{1}{2} \sum_{j=0}^{{J_{max}}-1} b_{j+\frac 12}u_{j}^2
\nn \\
& = & - \frac{1}{2} \sum_{j=1}^{{J_{max}}-1} (b_{j+\frac 12} - b_{j-\frac 12}) u_j^2  .
\label{e3-3}
\eey
Inserting (\ref{e3-2}) and (\ref{e3-3}) into (\ref{e3-1}), we get
\bey
&& \V{u}^T (A + \sqrt{M} \frac{d }{d t} \sqrt{M}) \V{u}
\nn \\
& = & - \sum_{j=1}^{{J_{max}}} a_{j-\frac 12}\frac{(u_{j}-u_{j-1})^2}{h_{j}} 
- \sum_{j=1}^{{J_{max}}-1} \frac{\dot{h}_{j+1}+\dot{h}_j}{4} u_j^2 - \sum_{j=1}^{{J_{max}}-1} \frac{h_{j+1}+h_j}{2} c_j u_j^2
\nn \\
& & - \frac{1}{2} \sum_{j=1}^{{J_{max}}-1} (b_{j+\frac 12} - b_{j-\frac 12}) u_j^2
+ \frac{1}{2} \sum_{j=1}^{{J_{max}}-1} (\dot{x}_{j+\frac 12} - \dot{x}_{j-\frac 12}) u_j^2
\nn \\
& = & - \sum_{j=1}^{{J_{max}}} a_{j-\frac 12}\frac{(u_{j}-u_{j-1})^2}{h_{j}} 
 - \sum_{j=1}^{{J_{max}}-1} \frac{h_{j+1}+h_j}{2} u_j^2
 \left [ c_j + \frac{1}{2} \frac{(b_{j+\frac 12} - b_{j-\frac 12})}{(h_{j+1}+h_j)/2} \right ] .
\nn
\eey 
The assumption (\ref{coef-3}) implies that the right-hand side is nonpositive. Thus,
scheme (\ref{fd-2}) satisfies the condition (\ref{cond-1}).
\end{proof}

\vspace{10pt}

\begin{rem}
\label{rem3.1}
The assumption (\ref{coef-3}) is a discrete analog to the continuous condition (\ref{coef-2}).
It is satisfied when $b$ is constant and $c$ is nonnegative. For the general situation with variable $b$,
it is reasonable to expect (\ref{coef-3}) to hold too provided that the continuous condition (\ref{coef-2}) holds and
the mesh is sufficiently fine.
\hfill \qed
\end{rem}

\vspace{10pt}

It is instructive to spell out one of the full discrete schemes. We consider the case with $m=1$ which results from
the application of the midpoint discretization to the $\V{v}$ equation (\ref{ode-v-1}). Let
\[
m_j(t) = \sqrt{\frac{h_{j+1}(t)+h_j(t)}{2}}.
\]
Then, the new and old variables are related by $v_j = m_j u_j$
and (\ref{fd-2}) reads as
\bey
&& m_j \dot{v_j}  + \frac{\dot{h}_{j+1}+\dot{h}_j}{4  }\frac{v_j}{m_j} + \frac{1}{2}
(b_{j+\frac 12}-\dot{x}_{j+\frac 12})\left (\frac{v_{j+1}}{m_{j+1}} +\frac{v_j}{m_{j}} \right )
\nn \\
&& \qquad  - \frac{1}{2} (b_{j-\frac 12}-\dot{x}_{j-\frac 12})\left (\frac{v_j}{m_{j}} 
+\frac{v_{j-1}}{m_{j-1}} \right ) + m_j c_j v_j  \nn \\
&&\qquad \qquad   = \;  \frac{a_{j+\frac 12}}{h_{j+1}}
\left (\frac{v_{j+1}}{m_{j+1}}  - \frac{v_j}{m_{j}} \right )
- \frac{a_{j-\frac 12}}{h_j} \left (\frac{v_j}{m_{j}} -\frac{v_{j-1}}{m_{j-1}} \right ) + m_j f_j .
\label{fd-2+1}
\eey
Applying the midpoint discretization to the above equation, we get
\bey
&& m_j^{n+\frac 1 2}\; \frac{v_j^{n+1}-v_j^n}{\Delta t_n}
 + \frac{\dot{h}_{j+1}^{n+\frac1 2}+\dot{h}_j^{n+\frac 1 2}}{8 }\; \frac{v_j^n+v_j^{n+1}}{m_j^{n+\frac 1 2}} 
  + \frac{1}{4}(b_{j+\frac 12}^{n+\frac 1 2}-\dot{x}_{j+\frac 12}^{n+\frac 1 2})
\left (\frac{v_{j+1}^n+v_{j+1}^{n+1}}{m_{j+1}^{n+\frac 1 2}} + \frac{v_{j}^n+v_{j}^{n+1}}{m_j^{n+\frac 1 2}}\right )
\nn \\
&& \qquad  - \frac{1}{4} (b_{j-\frac 12}^{n+\frac 1 2}-\dot{x}_{j-\frac 12}^{n+\frac 1 2} )\left (
\frac{v_{j}^n+v_{j}^{n+1}}{m_j^{n+\frac 1 2}} + \frac{v_{j-1}^n+v_{j-1}^{n+1}}{m_{j-1}^{n+\frac 1 2}}\right )
+ \frac 1 2 m_j^{n+\frac 1 2} c_j^{n+\frac 1 2} (v_j^n+v_j^{n+1})  \nn \\
&&\qquad \qquad   = \;  \frac{a_{j+\frac 12}^{n+\frac 1 2}}{2 h_{j+1}^{n+\frac 1 2}}
\left (\frac{v_{j+1}^n + v_{j+1}^{n+1}}{m_{j+1}^{n+\frac 1 2}}  - \frac{v_j^n + v_j^{n+1}}{m_j^{n+\frac 1 2}} \right ) 
- \frac{a_{j-\frac 12}^{n+\frac 1 2}}{2 h_j^{n+\frac 1 2}} \left (\frac{v_j^n+v_j^{n+1}}{m_j^{n+\frac 1 2}}
-\frac{v_{j-1}^n+v_{j-1}^{n+1}}{m_{j-1}^{n+\frac 1 2}} \right ) 
\nn \\
&&\qquad \qquad \qquad   +\;  m_j^{n+\frac 1 2} f_j^{n+\frac 1 2}  .
\label{fd-2+2}
\eey
This scheme is unconditionally stable and of second order in both time and space.

Next, we consider finite difference schemes based on the non-conservative form (\ref{pde-4}).
Approximating the mesh movement related convection term using the half point fluxes, i.e.,
\bey
\left.  \dot{x} \frac{\p u}{\p \xi}\right |_{x_j} & \approx &
\left.  \frac{1}{2} \dot{x} \frac{\p u}{\p \xi}\right |_{x_{j+\frac 1 2}}
\left. + \frac{1}{2} \dot{x} \frac{\p u}{\p \xi}\right |_{x_{j-\frac 1 2}}
\nn \\
& \approx &
 \frac{1}{2} \dot{x}_{j+\frac 12} (u_{j+1}-u_{j})
+ \frac{1}{2} \dot{x}_{j-\frac 12} (u_{j}-u_{j-1}) ,
\nn
\eey
we have
\bey
&& \frac{h_{j+1}+h_j}{2} \dot{u}_j 
- \frac{1}{2} \dot{x}_{j+\frac 12} (u_{j+1}-u_{j}) - \frac{1}{2} \dot{x}_{j-\frac 12} (u_{j}-u_{j-1})
\nn \\
&& \qquad +\; b_{j+\frac 12}\frac{u_{j+1}+u_j}{2} 
- b_{j-\frac 12}\frac{u_{j}+u_{j-1}}{2}  + \frac{h_{j+1}+h_j}{2} c_j u_j
\nn \\
&& \qquad \qquad = \; a_{j+\frac 12}\frac{u_{j+1}-u_{j}}{h_{j+1}}
- a_{j-\frac 12}\frac{u_{j}-u_{j-1}}{h_{j}}
+ \frac{h_{j+1}+h_j}{2} f_j  .
\label{fd-3}
\eey
Interestingly, it can be verified that this semi-dicrete scheme is equivalent to scheme (\ref{fd-2}). Thus,
(\ref{fd-3}) can be cast in the form (\ref{ode-3}) with property (\ref{cond-1}).

When a two-cell central finite difference approximation is used for the mesh movement related convection term,
the finite difference scheme becomes
\bey
&& \frac{h_{j+1}+h_j}{2} \dot{u}_j  - \frac{1}{2} \dot{x}_{j} (u_{j+1}-u_{j-1})
+ b_{j+\frac 12}\frac{u_{j+1}+u_j}{2} 
- b_{j-\frac 12}\frac{u_{j}+u_{j-1}}{2}  \nn \\
&&\qquad  + \; \frac{h_{j+1}+h_j}{2} c_j u_j = a_{j+\frac 12}\frac{u_{j+1}-u_{j}}{h_{j+1}}
- a_{j-\frac 12}\frac{u_{j}-u_{j-1}}{h_{j}} + \frac{h_{j+1}+h_j}{2} f_j  .
\label{fd-4}
\eey
It can be shown that
\bey
\V{u}^T (A + \sqrt{M} \frac{d }{d t} \sqrt{M}) \V{u} 
& = & - \sum_{j=1}^{{J_{max}}} \left [ \frac{a_{j-\frac 12}}{h_{j}} - \frac{(\dot{x}_j - \dot{x}_{j-1})}{4}\right ]  (u_j-u_{j-1})^2
\nn \\
&&  -\; \sum_{j=1}^{{J_{max}}-1} \frac{h_{j+1}+h_j}{2} u_j^2
 \left [ c_j + \frac{1}{2} \frac{(b_{j+\frac 12} - b_{j-\frac 12})}{(h_{j+1}+h_j)/2} \right ] .
\label{e3-5}
\eey
Thus, the right-hand side is nonpositive and therefore scheme (\ref{fd-4}) satisfies (\ref{cond-1}) if there hold
the conditions (\ref{coef-3}) and
\beq
(\dot{x}_j - \dot{x}_{j-1}) \le 4 \; \frac{a_{j-\frac 12}}{h_{j}}, \quad j = 1, ..., J_{max}-1 .
\label{cond-2}
\eeq
Note that both (\ref{coef-3}) and (\ref{cond-2}) can be satisfied when the mesh speed is bounded
and the mesh is sufficiently fine.

\section{2D convection-diffusion equations}
\label{SEC:2d}

In this section we study finite difference discretization of IBVP (\ref{ibvp}) in two dimensions.
The procedure is similar to that in the previous section for one dimension but
the derivation is more complicated for the current situation.

We assume that a curvilinear moving mesh, $\{ (x_{j,k}^n, y_{j,k}^n), j = 0, ..., J_{max}, k = 0, ..., K_{max}\}$,
is given for the physical domain $\Omega$ at time instants $t_0, ..., t_N$. As for the 1D case, we
consider the mesh to vary linearly on each time interval, i.e.,
\beq
x_{j,k}(t) = \frac{t_{n+1}-t}{\Delta t_n} x_{j,k}^n + \frac{t - t_n}{\Delta t_n} x_{j,k}^{n+1}, \quad
y_{j,k}(t) = \frac{t_{n+1}-t}{\Delta t_n} y_{j,k}^n + \frac{t - t_n}{\Delta t_n} y_{j,k}^{n+1} ,\quad
\forall t \in [t_n, t_{n+1}].
\label{mesh2d-1+1}
\eeq
Note that the mesh speed $(\dot{x}_{j,k},\dot{y}_{j,k})$ is constant on $[t_n, t_{n+1}]$.
We view the the moving mesh as the image of a Cartesian computational mesh
under  an invertible coordinate transformation $x = x(\xi, \eta, t), \;
y = y(\xi, \eta, t)$: $[0,J_{max}]\times [0,K_{max}] \to \Omega$, i.e.,
\beq
x_{j,k}(t) = x(\xi_j, \eta_k, t), \quad y_{j,k}(t) = y(\xi_j, \eta_k, t), \quad j = 0, ..., J_{max}, \; k = 0, ..., K_{max}
\label{mesh2d-1}
\eeq
where $\xi_j = j,\; j = 0,..., J_{max}$ and $\eta_k = k,\; k = 0, ..., K_{max}.$
Through the coordinate transformation, the PDE in (\ref{ibvp}) can be transformed (e.g., see \cite{HR11})
into a conservative form
\beq
J \dot{u} + u \dot{J} + \frac{\p q_1}{\p \xi} + \frac{\p q_2}{\p \eta}
+ c u J = \frac{\p p_1}{\p \xi} + \frac{\p p_2}{\p \eta} + J f ,
\label{e4-1}
\eeq
where
\beq
\begin{cases}
q_1 & =  u \left [ (J\xi_x) (b_1 - \dot{x}) + (J\xi_y)(b_2-\dot{y}) \right ] ,
\\
q_2 & =  u \left [ (J\eta_x) (b_1 - \dot{x}) + (J\eta_y)(b_2-\dot{y}) \right ] ,
\\
p_1 & =  \frac{a}{J}\left [(J\xi_x)^2 + (J\xi_y)^2\right ] \frac{\p u}{\p \xi}
+ \frac{a}{J} \left [(J\xi_x) (J\eta_x) + (J\xi_y) (J\eta_y) \right ] \frac{\p u}{\p \eta},
\\
p_2 & =  \frac{a}{J} \left [(J\xi_x) (J\eta_x) + (J\xi_y) (J\eta_y) \right ]\frac{\p u}{\p \xi}
+ \frac{a}{J} \left [(J\eta_x)^2+ (J\eta_y)^2 \right ] \frac{\p u}{\p \eta} .
\end{cases}
\label{e4-2}
\eeq
Here, $\V{b} = (b_1, b_2)^T$, $J$ is the Jacobian of the coordinate transformation, and $\V{q} = (q_1, q_2)^T$
and $\V{p} = (p_1, p_2)^T$ are the convection and diffusion fluxes, respectively.
Moreover, we have the transformation identities
\beq
\begin{cases}
J =  x_\xi y_\eta - x_\eta y_\xi = (J\xi_x) (J \eta_y) - (J \xi_y) (J \eta_x),\\
(J\xi_x) = y_\eta,\quad (J\xi_y) = - x_\eta,\quad
(J \eta_x) = - y_\xi,\quad (J \eta_y) = x_\xi .
\end{cases}
\label{e4-3}
\eeq

We consider a central finite difference discretization for (\ref{e4-1}). The scheme reads as
\bey
&& J_{j,k} \dot{u}_{j,k} + \dot{J}_{j,k} u_{j,k}
+ \left (q_{1,j+\frac 1 2,k} - q_{1,j-\frac 1 2,k}\right )
+ \left (q_{2,j,k+\frac 1 2} - q_{2,j,k-\frac 1 2}\right )
+ c_{j,k} u_{j,k} J_{j,k}
\nn \\
&& \qquad = \; \frac{1}{2} \left (p_{1,j+\frac 1 2,k+\frac 1 2} - p_{1,j-\frac 1 2,k+\frac 1 2}
+p_{1,j+\frac 1 2,k-\frac 1 2} - p_{1,j-\frac 1 2,k-\frac 1 2}\right )
\nn \\
&& \qquad \qquad +\; \frac{1}{2} \left (p_{2,j+\frac 1 2,k+\frac 1 2} - p_{2,j+\frac 1 2,k-\frac 1 2}
+p_{2,j-\frac 1 2,k+\frac 1 2} - p_{2,j-\frac 1 2,k-\frac 1 2}\right ) + f_{j,k} J_{j,k},
\label{fd-2d-1}
\eey
where the convection fluxes $q_1$ and $q_2$
are approximated at integer-half and half-integer points and
the diffusion fluxes are approximated at half-half points; i.e.,
\bey
q_{1, j-\frac 1 2,k} & = & \frac{u_{j,k}+u_{j-1,k}}{2} \left [ (J \xi_x)_{j-\frac 1 2,k} (b_{1, j-\frac 1 2,k}-\dot{x}_{j-\frac 1 2,k})
\frac{}{} \right .
\nn \\
&& \qquad \quad \qquad \left. \frac{}{}  + (J \xi_y)_{j-\frac 1 2,k} (b_{2, j-\frac 1 2,k}-\dot{y}_{j-\frac 1 2,k}) \right ],
\label{e4-5}
\\
q_{2, j,k-\frac 1 2} & = & \frac{u_{j,k}+u_{j,k-1}}{2} \left [ (J \eta_x)_{j,k-\frac 1 2} (b_{1, j,k-\frac 1 2}-\dot{x}_{j,k-\frac 1 2})
\frac{}{} \right .
\nn \\
&& \qquad \quad \qquad \left. \frac{}{} 
+ (J \eta_y)_{j,k-\frac 1 2} (b_{2, j,k-\frac 1 2}-\dot{y}_{j,k-\frac 1 2}) \right ],
\label{e4-6}
\\
p_{1, j-\frac{1}{2},k-\frac{1}{2}} & = & \frac{a_{j-\frac{1}{2},k-\frac{1}{2}}}{2 J_{j-\frac{1}{2},k-\frac{1}{2}}}
\left [ (J\xi_x)^2 + (J\xi_y)^2\right ]_{j-\frac{1}{2},k-\frac{1}{2}}
\left (u_{j,k}-u_{j-1,k}+u_{j,k-1}-u_{j-1,k-1}\right )
\nn \\
& + & \frac{a_{j-\frac{1}{2},k-\frac{1}{2}}}{2 J_{j-\frac{1}{2},k-\frac{1}{2}}}
\left [ (J\xi_x) (J\eta_x) + (J\xi_y)(J\eta_y)\right ]_{j-\frac{1}{2},k-\frac{1}{2}}
\nn \\
&& \qquad \qquad \qquad \times
\left (u_{j,k}-u_{j,k-1}+u_{j-1,k}-u_{j-1,k-1}\right ),
\qquad
\label{e4-6+1}
\\
p_{2, j-\frac{1}{2},k-\frac{1}{2}} & = & \frac{a_{j-\frac{1}{2},k-\frac{1}{2}}}{2 J_{j-\frac{1}{2},k-\frac{1}{2}}}
\left [ (J\xi_x) (J\eta_x) + (J\xi_y)(J\eta_y)\right ]_{j-\frac{1}{2},k-\frac{1}{2}}
\nn \\
&& \qquad \qquad \qquad \times
\left (u_{j,k}-u_{j-1,k}+u_{j,k-1}-u_{j-1,k-1}\right )
\nn \\
& + & \frac{a_{j-\frac{1}{2},k-\frac{1}{2}}}{2 J_{j-\frac{1}{2},k-\frac{1}{2}}}
\left [ (J\eta_x)^2 + (J\eta_y)^2\right ]_{j-\frac{1}{2},k-\frac{1}{2}}
\left (u_{j,k}-u_{j,k-1}+u_{j-1,k}-u_{j-1,k-1}\right ) .
\qquad
\label{e4-6+2}
\eey
Here, the transformation quantities (used in the approximations of the convection fluxes)
\beq
(J \xi_x)_{j-\frac 1 2,k},\quad (J \xi_y)_{j-\frac 1 2,k},\quad (J \eta_x)_{j,k-\frac 1 2},\quad 
(J \eta_y)_{j,k-\frac 1 2},\quad \dot{x}_{j-\frac 1 2,k},\quad \dot{x}_{j,k-\frac 1 2},\quad \dot{y}_{j-\frac 1 2,k},\quad
\dot{y}_{j,k-\frac 1 2},
\label{e4-6+3}
\eeq
are to be defined so that condition (\ref{cond-1}) is satisfied while the other quantities
(used in the approximations of the diffusion fluxes)
\beq
(J \xi_x)_{j-\frac 1 2,k-\frac 1 2},\quad (J \xi_y)_{j-\frac 1 2,k-\frac 1 2},\quad
(J \eta_x)_{j-\frac 1 2,k-\frac 1 2},\quad (J \eta_y)_{j-\frac 1 2,k-\frac 1 2},\quad J_{j-\frac{1}{2},k-\frac{1}{2}},
\label{e4-6+4}
\eeq
are approximated using central finite differences based on relation (\ref{e4-3}). For example,
\[
(J \xi_x)_{j-\frac 1 2,k-\frac 1 2} = (y_\eta)_{j-\frac 1 2,k-\frac 1 2}
\approx \frac 1 2 (y_{j, k} - y_{j,k-1} + y_{j-1, k} - y_{j-1, k-1}) .
\]
The discretization of the boundary condition is
\beq
u_{j, k} = g(x_{j,k}(t), y_{j,k}(t), t), \quad \forall (j,k) \mbox{ with } j = 0, \; j = J_{max},\; k = 0, \mbox{ or }k = K_{max}.
\label{fd-2d-2}
\eeq
Let $\V{u} = \{ u_{j,k} \}$. The above scheme can then be cast in the form (\ref{ode-3}) with
\bey
(M \V{u})_{(j,k)} & = & J_{j,k} \dot{u}_{j,k} ,
\label{fd-2d-3}
\\
(A \V{u})_{(j,k)} & = & - \dot{J}_{j,k} u_{j,k}
- \left (q_{1,j+\frac 1 2,k} - q_{1,j-\frac 1 2,k}\right ) - \left (q_{2,j,k+\frac 1 2} - q_{2,j,k-\frac 1 2}\right )
- c_{j,k} u_{j,k} J_{j,k}
\nn \\
& & + \; \frac{1}{2} \left (f_{1,j+\frac 1 2,k+\frac 1 2} - f_{1,j-\frac 1 2,k+\frac 1 2}+f_{1,j+\frac 1 2,k-\frac 1 2} - f_{1,j-\frac 1 2,k-\frac 1 2}\right )
\nn \\
&& +\; \frac{1}{2} \left (f_{2,j+\frac 1 2,k+\frac 1 2} - f_{2,j+\frac 1 2,k-\frac 1 2}+f_{2,j-\frac 1 2,k+\frac 1 2} - f_{2,j-\frac 1 2,k-\frac 1 2}\right ) .
\label{fd-2d-4}
\eey
Note that the mass matrix is diagonal, with the diagonal entries being $J_{j,k}$.

\begin{thm}
\label{thm4.1}
Assume that for $j = 1,...,J_{max}-1$ and $k = 1, ..., K_{max}-1$, there hold
\bey
J_{j,k} c_{j,k}
& + &  \frac{1}{2} \left [ 
(J\xi_x)_{j+\frac 1 2,k} b_{1,j+\frac 1 2,k} - (J\xi_x)_{j-\frac 1 2,k} b_{1,j-\frac 1 2,k} 
+ (J\xi_y)_{j+\frac 1 2,k} b_{2,j+\frac 1 2,k} - (J\xi_y)_{j-\frac 1 2,k} b_{2,j-\frac 1 2,k} \right ]
\nn \\
& + & \frac{1}{2} \left [ (J\eta_x)_{j,k+\frac 1 2} b_{1,j,k+\frac 1 2}
- (J\eta_x)_{j,k-\frac 1 2} b_{1,j,k-\frac 1 2} + (J\eta_y)_{j,k+\frac 1 2} b_{2,j,k+\frac 1 2}
- (J\eta_y)_{j,k-\frac 1 2} b_{2,j,k-\frac 1 2} \right ] 
\nn \\
& \ge & 0 ,
\label{coef-4}
\\
\dot{J}_{j,k} & = &\quad  
 (J\xi_x)_{j+\frac 1 2,k} \dot{x}_{j+\frac 1 2,k} - (J\xi_x)_{j-\frac 1 2,k} \dot{x}_{j-\frac 1 2,k} 
 +(J\xi_y)_{j+\frac 1 2,k} \dot{y}_{j+\frac 1 2,k} - (J\xi_y)_{j-\frac 1 2,k} \dot{y}_{j-\frac 1 2,k} 
\nn \\
&& +\; (J\eta_x)_{j,k+\frac 1 2} \dot{x}_{j,k+\frac 1 2} - (J\eta_x)_{j,k-\frac 1 2} \dot{x}_{j,k-\frac 1 2}
+ (J\eta_y)_{j,k+\frac 1 2} \dot{y}_{j,k+\frac 1 2} - (J\eta_y)_{j,k-\frac 1 2} \dot{y}_{j,k-\frac 1 2} .
\label{gcl-1}
\eey
Then, the finite difference scheme (\ref{fd-2d-1}) satisfies the condition (\ref{cond-1}).
 As a consequence, the fully discrete scheme resulting from
the application of the time integration method (\ref{col-1}) to (\ref{ode-3})
with $M$ and $A$ given in (\ref{fd-2d-3}) and (\ref{fd-2d-4})
is unconditionally stable and of order $(2m)$ in time and order 2 in space.
\end{thm}

\begin{proof}
Once again, we take $f \equiv 0$ and $g \equiv 0$ for stability analysis.
From (\ref{fd-2d-2}), (\ref{fd-2d-3}), and (\ref{fd-2d-4}),  we have 
\bey
&& \V{u}^T (A + \sqrt{M} \frac{d }{d t} \sqrt{M} ) \V{u}
\nn \\
& = & - \frac{1}{2} \sum_{j=1}^{{J_{max}}-1} \sum_{k=1}^{{K_{max}}-1} \dot{J}_{j,k} u_{j,k}^2
- \sum_{j=1}^{{J_{max}}-1} \sum_{k=1}^{{K_{max}}-1} c_{j,k}  J_{j,k} u_{j,k}^2
\nn \\
&& - \sum_{j=1}^{{J_{max}}-1}\sum_{k=1}^{{K_{max}}-1} u_{j,k}
\left [ \left (q_{1,j+\frac 1 2,k} - q_{1,j-\frac 1 2,k}\right ) + \left (q_{2,j,k+\frac 1 2} - q_{2,j,k-\frac 1 2}\right )\right ]
\nn \\
&& + \frac{1}{2} \sum_{j=1}^{{J_{max}}-1}\sum_{k=1}^{{K_{max}}-1} u_{j,k}
\left (f_{1,j+\frac 1 2,k+\frac 1 2} - f_{1,j-\frac 1 2,k+\frac 1 2}+f_{1,j+\frac 1 2,k-\frac 1 2} - f_{1,j-\frac 1 2,k-\frac 1 2}\right )
\nn \\
&& + \frac{1}{2} \sum_{j=1}^{{J_{max}}-1}\sum_{k=1}^{{K_{max}}-1} u_{j,k}
\left (f_{2,j+\frac 1 2,k+\frac 1 2} - f_{2,j+\frac 1 2,k-\frac 1 2}+f_{2,j-\frac 1 2,k+\frac 1 2} - f_{2,j-\frac 1 2,k-\frac 1 2}\right ) 
\nn \\
& = & - \frac{1}{2} \sum_{j=1}^{{J_{max}}-1} \sum_{k=1}^{{K_{max}}-1} \dot{J}_{j,k} u_{j,k}^2
- \sum_{j=1}^{{J_{max}}-1} \sum_{k=1}^{{K_{max}}-1} c_{j,k}  J_{j,k} u_{j,k}^2
\nn \\
&& - \sum_{j=1}^{{J_{max}}}\sum_{k=1}^{{K_{max}}-1} q_{1,j-\frac 1 2,k} \left (u_{j-1,k} - u_{j,k}\right )
- \sum_{j=1}^{{J_{max-1}}}\sum_{k=1}^{{K_{max}}} q_{2,j,k-\frac 1 2} \left (u_{j,k-1} - u_{j,k}\right )
\nn \\
&&  - \frac{1}{2} \sum_{j=1}^{{J_{max}}} \sum_{k=1}^{{K_{max}}} f_{1,j-\frac 1 2,k-\frac 1 2} 
(u_{j,k}-u_{j-1,k} + u_{j,k-1} - u_{j-1,k-1})
\nn \\
&& - \frac{1}{2}  \sum_{j=1}^{{J_{max}}} \sum_{k=1}^{{K_{max}}} f_{2,j-\frac 1 2,k-\frac 1 2} 
(u_{j,k}-u_{j,k-1} + u_{j-1,k} - u_{j-1,k-1}).
\label{e4-4}
\eey
From (\ref{e4-5}) and (\ref{e4-6}), the convection-related terms (the third and fourth terms on the right-hand
side of the above equality) become
\bey
&& \sum_{j=1}^{{J_{max}}}\sum_{k=1}^{{K_{max}}-1} q_{1,j-\frac 1 2,k} \left (u_{j-1,k} - u_{j,k}\right )
\nn \\
& & \quad = \; \frac{1}{2} \sum_{j=1}^{{J_{max}}-1}\sum_{k=1}^{{K_{max}}-1} u_{j,k}^2
\left [ \frac{}{} \right . (J \xi_x)_{j+\frac 1 2,k} (b_{1, j+\frac 1 2,k}-\dot{x}_{j+\frac 1 2,k})
\nn \\
&& \qquad + (J \xi_y)_{j+\frac 1 2,k} (b_{2, j+\frac 1 2,k}-\dot{y}_{j+\frac 1 2,k}) - (J \xi_x)_{j-\frac 1 2,k} (b_{1, j-\frac 1 2,k}-\dot{x}_{j-\frac 1 2,k})
\nn \\
&& \qquad  - (J \xi_y)_{j-\frac 1 2,k} (b_{2, j-\frac 1 2,k}-\dot{y}_{j-\frac 1 2,k}) \left. \frac{}{} \right ],
\label{e4-7}
\eey
and
\bey
&& \sum_{j=1}^{{J_{max-1}}}\sum_{k=1}^{{K_{max}}} q_{2,j,k-\frac 1 2} \left (u_{j,k-1} - u_{j,k}\right )
\nn \\
& & \quad = \; \frac{1}{2} \sum_{j=1}^{{J_{max}}-1}\sum_{k=1}^{{K_{max}}-1} u_{j,k}^2
\left [ \frac{}{} \right . (J \eta_x)_{j,k+\frac 1 2} (b_{1, j,k+\frac 1 2}-\dot{x}_{j,k+\frac 1 2})
\nn \\
&& \qquad + (J \eta_y)_{j,k+\frac 1 2} (b_{2, j,k+\frac 1 2}-\dot{y}_{j,k+\frac 1 2}) - (J \eta_x)_{j,k-\frac 1 2} (b_{1, j,k-\frac 1 2}-\dot{x}_{j,k-\frac 1 2})
\nn \\
&& \qquad  - (J \eta_y)_{j,k-\frac 1 2} (b_{2, j,k-\frac 1 2}-\dot{y}_{j,k-\frac 1 2}) \left. \frac{}{} \right ] .
\label{e4-8}
\eey
Moreover, from (\ref{e4-6+1}) and (\ref{e4-6+2}) the diffusion terms in (\ref{e4-4}) can be written as
\bey
&&  - \frac{1}{2} \sum_{j=1}^{{J_{max}}} \sum_{k=1}^{{K_{max}}} f_{1,j-\frac 1 2,k-\frac 1 2} 
(u_{j,k}-u_{j-1,k} + u_{j,k-1} - u_{j-1,k-1})
\nn \\
&& \;\; - \frac{1}{2}  \sum_{j=1}^{{J_{max}}} \sum_{k=1}^{{K_{max}}} f_{2,j-\frac 1 2,k-\frac 1 2} 
(u_{j,k}-u_{j,k-1} + u_{j-1,k} - u_{j-1,k-1})
\nn \\
& & \quad = - \; \frac{1}{4} \sum_{j=1}^{{J_{max}}} \sum_{k=1}^{{K_{max}}} \frac{a_{j-\frac 1 2,k-\frac 1 2}}{J_{j-\frac 1 2,k-\frac 1 2}}
\left [ (J\xi_x)_{j-\frac 1 2,k-\frac 1 2} (u_{j,k}-u_{j-1,k} + u_{j,k-1} - u_{j-1,k-1})\frac{}{} \right .
\nn \\
&& \qquad \qquad \qquad \qquad \qquad \left. \frac{}{}
+ (J\eta_x)_{j-\frac 1 2,k-\frac 1 2} (u_{j,k}-u_{j,k-1} + u_{j-1,k} - u_{j-1,k-1})\right ]^2
\nn \\
&& \quad \quad - \; \frac{1}{4} \sum_{j=1}^{{J_{max}}} \sum_{k=1}^{{K_{max}}} \frac{a_{j-\frac 1 2,k-\frac 1 2}}{J_{j-\frac 1 2,k-\frac 1 2}}
\left [ (J\xi_y)_{j-\frac 1 2,k-\frac 1 2} (u_{j,k}-u_{j-1,k} + u_{j,k-1} - u_{j-1,k-1})\frac{}{} \right .
\nn \\
&& \qquad \qquad \qquad \qquad \qquad \left. \frac{}{}
+ (J\eta_y)_{j-\frac 1 2,k-\frac 1 2} (u_{j,k}-u_{j,k-1} + u_{j-1,k} - u_{j-1,k-1})\right ]^2 .
\label{e4-9}
\eey
Combining (\ref{e4-7})--(\ref{e4-9}) with (\ref{e4-4}), we obtain
\bey
&& \V{u}^T (A + \sqrt{M} \frac{d }{d t} \sqrt{M} ) \V{u}
\nn \\
& = & - \; \frac{1}{4} \sum_{j=1}^{{J_{max}}} \sum_{k=1}^{{K_{max}}} \frac{a_{j-\frac 1 2,k-\frac 1 2}}{J_{j-\frac 1 2,k-\frac 1 2}}
\left [ (J\xi_x)_{j-\frac 1 2,k-\frac 1 2} (u_{j,k}-u_{j-1,k} + u_{j,k-1} - u_{j-1,k-1})\frac{}{} \right .
\nn \\
&& \qquad \qquad \qquad \qquad \qquad \left. \frac{}{}
+ (J\eta_x)_{j-\frac 1 2,k-\frac 1 2} (u_{j,k}-u_{j,k-1} + u_{j-1,k} - u_{j-1,k-1})\right ]^2
\nn \\
&&  - \; \frac{1}{4} \sum_{j=1}^{{J_{max}}} \sum_{k=1}^{{K_{max}}} \frac{a_{j-\frac 1 2,k-\frac 1 2}}{J_{j-\frac 1 2,k-\frac 1 2}}
\left [ (J\xi_y)_{j-\frac 1 2,k-\frac 1 2} (u_{j,k}-u_{j-1,k} + u_{j,k-1} - u_{j-1,k-1})\frac{}{} \right .
\nn \\
&& \qquad \qquad \qquad \qquad \qquad \left. \frac{}{}
+ (J\eta_y)_{j-\frac 1 2,k-\frac 1 2} (u_{j,k}-u_{j,k-1} + u_{j-1,k} - u_{j-1,k-1})\right ]^2
\nn \\
& & - \frac{1}{2}\sum_{j=1}^{{J_{max}}-1}\sum_{k=1}^{{K_{max}}-1} u_{j,k}^2 \left [ \dot{J}_{j,k}
- (J\xi_x)_{j+\frac 1 2,k} \dot{x}_{j+\frac 1 2,k} + (J\xi_x)_{j-\frac 1 2,k} \dot{x}_{j-\frac 1 2,k} \frac{}{} \right.
\nn \\
&& \qquad \qquad - (J\xi_y)_{j+\frac 1 2,k} \dot{y}_{j+\frac 1 2,k} + (J\xi_y)_{j-\frac 1 2,k} \dot{y}_{j-\frac 1 2,k}
- (J\eta_x)_{j,k+\frac 1 2} \dot{x}_{j,k+\frac 1 2} \frac{}{}
\nn \\
&& \qquad \qquad \left. + (J\eta_x)_{j,k-\frac 1 2} \dot{x}_{j,k-\frac 1 2} - (J\eta_y)_{j,k+\frac 1 2} \dot{y}_{j,k+\frac 1 2}
+ (J\eta_y)_{j,k-\frac 1 2} \dot{y}_{j,k-\frac 1 2} \frac{}{} \right ] 
\nn \\
&& - \frac{1}{2}\sum_{j=1}^{{J_{max}}-1}\sum_{k=1}^{{K_{max}}-1} u_{j,k}^2 \left [ 2 c_{j,k} J_{j,k}
+ (J\xi_x)_{j+\frac 1 2,k} b_{1,j+\frac 1 2,k} - (J\xi_x)_{j-\frac 1 2,k} b_{1,j-\frac 1 2,k} \frac{}{} \right.
\nn \\
&& \qquad \qquad + (J\xi_y)_{j+\frac 1 2,k} b_{2,j+\frac 1 2,k} - (J\xi_y)_{j-\frac 1 2,k} b_{2,j-\frac 1 2,k}
+ (J\eta_x)_{j,k+\frac 1 2} b_{1,j,k+\frac 1 2} \frac{}{}
\nn \\
&& \qquad  \qquad \left. - (J\eta_x)_{j,k-\frac 1 2} b_{1,j,k-\frac 1 2} + (J\eta_y)_{j,k+\frac 1 2} b_{2,j,k+\frac 1 2}
- (J\eta_y)_{j,k-\frac 1 2} b_{2,j,k-\frac 1 2} \frac{}{} \right ] . 
\label{e4-10}
\eey
The right-hand side of the above inequality is nonpositive and therefore
scheme (\ref{fd-2d-1}) satisfies the condition (\ref{cond-1}) if (\ref{coef-4}) and (\ref{gcl-1}) hold.
\end{proof}

\begin{rem}
\label{rem4.1}
The condition (\ref{coef-4}) is a central finite difference approximation to
the condition (\ref{coef-2}) which takes the form in the new coordinates $\xi$ and $\eta$ as
\[
J c + \frac{1}{2} \frac{\partial}{\partial \xi} \left [ (J\xi_x) b_1 \frac{}{} - (J\xi_y) b_2 \right ]
+ \frac{1}{2} \frac{\partial}{\partial \eta} \left [ (J\eta_x) b_1 \frac{}{} + (J\eta_y) b_2 \right ] \ge 0.
\]
As mentioned in Remark~\ref{rem3.1}, (\ref{coef-4}) holds when $\V{b} = (b_1, b_2)^T$ is constant
and $c$ is nonnegative. For the general situation with variable $\V{b}$, it is reasonable to expect
the condition to hold provided that (\ref{coef-2}) holds and the mesh is sufficiently fine.
\hfill \qed
\end{rem}

\begin{rem}
\label{rem4.2}
The condition (\ref{gcl-1}) is new in multidimensions and often referred to as
a geometric conservation law (GCL) in the literature. As a matter of fact, it is a central finite
difference approximation of the continuous identity
\beq
\dot{J} = \dot{\overline{(J\xi_x)}} (J \eta_y) + (J\xi_x)\dot{\overline{ (J \eta_y)}}
 - \dot{\overline{ (J \xi_y)}} (J \eta_x) - (J \xi_y) \dot{\overline{(J \eta_x)}},
\label{gcl-2}
\eeq
where the symbol ``$\dot{\overline{\frac{}{}\;\frac{}{}}}$'' denotes the time derivative of a product.
The importance of satisfying GCLs by numerical algorithms has been extensively studied;
e.g., see \cite{BG04,FN99,FN04,Hin82,TL79}. Generally speaking, there are
two approaches to enforce GCL (\ref{gcl-1}). The first one,
proposed by Thomas and Lombard \cite{TL79},
is to consider $J_{j,k}$ as an unknown variable and update it by integrating (\ref{gcl-1})
using the same scheme as for the underlying PDE. The main advantage of this approach
is that the transformation quantities in (\ref{gcl-1}) (also see (\ref{e4-6+3}))
can be approximated using arbitrary finite differences.

The other approach is to choose proper approximations for
those transformation quantities in (\ref{e4-6+3}) such that (\ref{gcl-1}) is satisfied automatically.
An example set of such approximations is given by
\beq
\begin{cases}
\dot{x}_{j- \frac 1 2,k} & \equiv  \frac{1}{8} \left [\dot{x}_{j,k-1}+\dot{x}_{j- 1,k-1}+2 \dot{x}_{j,k}+ 2 \dot{x}_{j- 1,k}
+ \dot{x}_{j,k+1}+\dot{x}_{j- 1,k+1} \right ] ,
\\
\dot{x}_{j,k- \frac 1 2} & \equiv  \frac{1}{8} \left [\dot{x}_{j-1,k}+\dot{x}_{j-1,k- 1}+2 \dot{x}_{j,k}+ 2 \dot{x}_{j,k- 1}
+ \dot{x}_{j+1,k}+\dot{x}_{j+1,k- 1} \right ],
\\
\dot{y}_{j- \frac 1 2,k} & \equiv  \frac{1}{8} \left [\dot{y}_{j,k-1}+\dot{y}_{j- 1,k-1}+2 \dot{y}_{j,k}+ 2 \dot{y}_{j- 1,k}
+ \dot{y}_{j,k+1}+\dot{y}_{j- 1,k+1} \right ],
\\
\dot{y}_{j,k- \frac 1 2} & \equiv  \frac{1}{8} \left [\dot{y}_{j-1,k}+\dot{y}_{j-1,k- 1}+2 \dot{y}_{j,k}+ 2 \dot{y}_{j,k- 1}
+ \dot{y}_{j+1,k}+\dot{y}_{j+1,k- 1} \right ] ,
\end{cases}
\label{e4-11}
\eeq
\beq
\begin{cases}
(J \xi_x)_{j- \frac 1 2,k} & =  (y_\eta)_{j- \frac 1 2,k} \equiv \frac{1}{4} [ y_{j,k+1}-y_{j,k-1} + y_{j- 1,k+1}-y_{j- 1,k-1}],
\\
(J\xi_y)_{j- \frac 1 2,k} & =  - (x_\eta)_{j- \frac 1 2,k} \equiv - \frac{1}{4}[ x_{j,k+1}-x_{j,k-1}+x_{j- 1, k+1}-x_{j- 1, k-1}] ,
\\
(J\eta_x)_{j,k- \frac 1 2} & =  - (y_\xi)_{j,k- \frac 1 2} \equiv - \frac{1}{4}[ y_{j+1,k}-y_{j-1,k} + y_{j+1,k- 1}-y_{j-1,k- 1}],
\\
(J\eta_y)_{j,k- \frac 1 2} & =  (x_\xi)_{j,k- \frac 1 2} \equiv \frac{1}{4} [ x_{j+1,k}-x_{j-1,k}+x_{j+1,k- 1}-x_{j-1,k- 1}] ,
\label{e4-12}
\end{cases}
\eeq
\bey
J_{j,k} & = & [(J\xi_x) (J \eta_y) - (J \xi_y) (J \eta_x)]_{j,k}
\nn \\
& \equiv & \frac{1}{4} [(J\xi_x)_{j+\frac 1 2,k} + (J\xi_x)_{j-\frac 1 2,k}] \; [(J\eta_y)_{j,k+\frac 1 2} + (J\eta_y)_{j,k-\frac 1 2} ]
\nn \\
&& -\; \frac{1}{4} [(J\xi_y)_{j+\frac 1 2,k}  + (J\xi_y)_{j-\frac 1 2,k} ]\;  [(J\eta_x)_{j,k+\frac 1 2}  + (J\eta_x)_{j,k-\frac 1 2}] . 
\label{e4-13}
\eey
To show that (\ref{gcl-1}) is satisfied, using the identity
\[
ac - bd = \frac{1}{2} (a-b)(c+d) + \frac{1}{2} (a+b) (c-d),
\]
we can rewrite (\ref{gcl-1}) into
\bey
\dot{J}_{j,k} & = & 
\frac{1}{2} [(J\xi_x)_{j+\frac 1 2,k} - (J\xi_x)_{j-\frac 1 2,k}] (\dot{x}_{j+\frac 1 2,k} + \dot{x}_{j-\frac 1 2,k})
\nn \\
& + & \frac{1}{2} [(J\xi_x)_{j+\frac 1 2,k} + (J\xi_x)_{j-\frac 1 2,k}] (\dot{x}_{j+\frac 1 2,k} - \dot{x}_{j-\frac 1 2,k})
\nn \\
&+ & \frac{1}{2} [(J\xi_y)_{j+\frac 1 2,k}  - (J\xi_y)_{j-\frac 1 2,k} ] (\dot{y}_{j+\frac 1 2,k} + \dot{y}_{j-\frac 1 2,k})
\nn \\
&+ & \frac{1}{2} [(J\xi_y)_{j+\frac 1 2,k}  + (J\xi_y)_{j-\frac 1 2,k} ] (\dot{y}_{j+\frac 1 2,k} - \dot{y}_{j-\frac 1 2,k})
\nn \\
&+ & \frac{1}{2} [(J\eta_x)_{j,k+\frac 1 2}  - (J\eta_x)_{j,k-\frac 1 2}] (\dot{x}_{j,k+\frac 1 2} + \dot{x}_{j,k-\frac 1 2})
\nn \\
&+ & \frac{1}{2} [(J\eta_x)_{j,k+\frac 1 2}  + (J\eta_x)_{j,k-\frac 1 2}] (\dot{x}_{j,k+\frac 1 2} - \dot{x}_{j,k-\frac 1 2})
\nn \\
&+ & \frac{1}{2} [(J\eta_y)_{j,k+\frac 1 2} - (J\eta_y)_{j,k-\frac 1 2} ] (\dot{y}_{j,k+\frac 1 2}  + \dot{y}_{j,k-\frac 1 2})
\nn \\
&+ & \frac{1}{2} [(J\eta_y)_{j,k+\frac 1 2} + (J\eta_y)_{j,k-\frac 1 2} ] (\dot{y}_{j,k+\frac 1 2}  - \dot{y}_{j,k-\frac 1 2}) .
\label{e4-14}
\eey
Notice that the approximations defined in (\ref{e4-11}) and (\ref{e4-12}) satisfy
\bey
&& \dot{x}_{j+\frac 1 2,k} + \dot{x}_{j-\frac 1 2,k} = \dot{x}_{j,k+\frac 1 2} + \dot{x}_{j,k-\frac 1 2},
\nn \\
&& \dot{y}_{j+\frac 1 2,k} + \dot{y}_{j-\frac 1 2,k} = \dot{y}_{j,k+\frac 1 2}  + \dot{y}_{j,k-\frac 1 2},
\nn \\
&& (J\xi_x)_{j+\frac 1 2,k} - (J\xi_x)_{j-\frac 1 2,k} + (J\eta_x)_{j,k+\frac 1 2}  - (J\eta_x)_{j,k-\frac 1 2} = 0,
\nn \\
&& (J\xi_y)_{j+\frac 1 2,k}  - (J\xi_y)_{j-\frac 1 2,k} + (J\eta_y)_{j,k+\frac 1 2} - (J\eta_y)_{j,k-\frac 1 2} = 0.
\nn
\eey
Inserting these identities into (\ref{e4-14}), we get
\bey
\dot{J}_{j,k} & = & 
 \frac{1}{2} [(J\xi_x)_{j+\frac 1 2,k} + (J\xi_x)_{j-\frac 1 2,k}] (\dot{x}_{j+\frac 1 2,k} - \dot{x}_{j-\frac 1 2,k})
\nn \\
&+ & \frac{1}{2} [(J\xi_y)_{j+\frac 1 2,k}  + (J\xi_y)_{j-\frac 1 2,k} ] (\dot{y}_{j+\frac 1 2,k} - \dot{y}_{j-\frac 1 2,k})
\nn \\
&+ & \frac{1}{2} [(J\eta_x)_{j,k+\frac 1 2}  + (J\eta_x)_{j,k-\frac 1 2}] (\dot{x}_{j,k+\frac 1 2} - \dot{x}_{j,k-\frac 1 2})
\nn \\
&+ & \frac{1}{2} [(J\eta_y)_{j,k+\frac 1 2} + (J\eta_y)_{j,k-\frac 1 2} ] (\dot{y}_{j,k+\frac 1 2}  - \dot{y}_{j,k-\frac 1 2}) .
\label{gcl-3}
\eey
Moreover, it can be shown that
\bey
&& \frac{d }{d t} \frac{1}{2} [(J\eta_y)_{j,k+\frac 1 2} + (J\eta_y)_{j,k-\frac 1 2} ] = (\dot{x}_{j+\frac 1 2,k} - \dot{x}_{j-\frac 1 2,k}),
\nn \\
&& \frac{d }{d t} \frac{1}{2} [(J\xi_x)_{j+\frac 1 2,k} + (J\xi_x)_{j-\frac 1 2,k}]= (\dot{y}_{j,k+\frac 1 2}  - \dot{y}_{j,k-\frac 1 2}) ,
\nn \\
&& \frac{d }{d t} \frac{1}{2} [(J\eta_x)_{j,k+\frac 1 2}  + (J\eta_x)_{j,k-\frac 1 2}] = - (\dot{y}_{j+\frac 1 2,k} - \dot{y}_{j-\frac 1 2,k}),
\nn \\
&& \frac{d }{d t} \frac{1}{2} [(J\xi_y)_{j+\frac 1 2,k}  + (J\xi_y)_{j-\frac 1 2,k} ] = - (\dot{x}_{j,k+\frac 1 2} - \dot{x}_{j,k-\frac 1 2}) .
\nn
\eey
Using this and (\ref{e4-13}) we can see that (\ref{gcl-3}), and therefore (\ref{gcl-1}), hold.
\hfill \qed
\end{rem}

\section{Numerical examples}
\label{SEC:numerics}

In this section we present numerical results obtained with the numerical schemes developed
in the previous sections for two one dimensional and one two dimensional examples. Our objective is
to verify the stability and accuracy of those schemes.

\begin{exam}
\label{exam5.1}
We first consider a one dimensional  example in the form (\ref{ibvp-1d}).  The coefficients
are given by
\beq
\begin{cases}
a(x,t) = 1,\quad b(x,t) = 0,\quad c(x,t) = 0,\\
x_l = 0,\quad x_r = \pi .
\end{cases}
\label{exam5.1-1}
\eeq
The source term $f(x,t)$, the initial solution, and the Dirichlet boundary conditions are chosen
such that the exact solution of the IBVP is given by
\beq
u_{exact}(x,t) = (2 + \sin(\pi t)) \sin(x)  .
\label{exam5.1-2}
\eeq
Notice that we have $g(x,t) \equiv 0$ for this exact solution.
The mesh is chosen as
\beq
x_j(t) = \frac{j \pi}{J_{max}} + \frac{1}{4} \sin\left (\frac{2 j \pi}{J_{max}}\right )  \sin(\omega t),\quad j = 0, ..., J_{max}
\label{exam5.1-3}
\eeq
where the parameter $\omega$ is used to control the speed of mesh movement.
The trajectories of two meshes with $\omega = 2\pi$ and $20 \pi$, respectively, are shown in Fig. \ref{f5-1-1}.

Fig. \ref{f5-1-2}(a) shows the maximum error as a function of $\Delta t$  obtained with scheme (\ref{col-1}) + (\ref{fd-2+1})
($m=1$) for mesh (\ref{exam5.1-3}) ($\omega = 2\pi$). In the computation, $J_{max}$ is taken sufficiently
large ($J_{max} = 1000$). In this case, the spatial discretization error is ignorable and the total error
is dominated by the temporal discretization error. The results show the $O(\Delta t^2)$ behavior of the error,
consistent with the theoretical prediction. They also demonstrate that the scheme is stable for all used values of $\Delta t$.

Fig. \ref{f5-1-2}(b) shows the maximum error as a function of $J_{max}$. The results are obtained
with $m = 1, 2, 3$ and the time step size $\Delta t = (\pi/J_{max})^{1/m}$, respectively. The reason $\Delta t$ is
chosen this way is that the error for scheme (\ref{fd-2+2}) is of order $O(\Delta t^{2m}) + O(\Delta x^2)$,
which reduces to $O(J_{max}^{-2})$ with this choice of $\Delta t$.
The figure confirms that the maximum error converges quadratically for all three cases.

The numerical results obtained with scheme (\ref{col-1}) + (\ref{fd-2+1}) ($m=1$) for a faster moving mesh $\omega = 20\pi$
are shown in Fig. \ref{f5-1-3}. Once again, the theoretically predicted order of convergence (second order
in both time and space for $m=1$) is observed. Interestingly, one may observe that there is a bump in the curve
in Fig. \ref{f5-1-3}(a). This is because for large $\Delta t$, the fast movement of the mesh is not ``felt'' by
the integration and this results in smaller error.

To conclude this example, we study the approximation of the boundary condition (\ref{bc-3})
with a nonhomogeneous term.  Applying the collocation scheme (\ref{col-1}) to (\ref{bc-3}) (for simplicity
we keep it in the old variable), we get
\beq
u_0({t}_{n,j}) = g(x_l({t}_{n,j}), {t}_{n,j}),\quad u_{J_{max}}({t}_{n,j}) = g(x_r({t}_{n,j}), {t}_{n,j}),\quad j = 1, ..., m
\label{bc-4}
\eeq
where $u_0(t) \approx u(x_0(t),t)$, $u_{J_{max}}(t) \approx u(x_{J_{max}}(t), t)$,
and $t_{n,1},..., t_{n,m}$ are the collocation points (cf. (\ref{col-points})).
This has the advantage that the boundary condition is collocated at the same points as the PDE and thus
can be implemented conveniently within the framework of the method of lines. 
On the other hand, since the boundary condition is not imposed at $t=t_{n+1}$,
error may occur there and accumulate and eventually result in inaccurate computational solutions.
This does not seem to be an issue for a homogeneous boundary condition.
However, the situation is different when a non-homogenous boundary condition is involved.

To see this, we consider a case where the same PDE and mesh (with $\omega = 2\pi$)
are used but the source term, initial condition, and (non-homogeneous) boundary condition
are chosen such that the exact solution is given by
\beq
u_{exact}(x,t) = (2 + \sin(\pi t)) \cos(x)  .
\label{exam5.1-4}
\eeq
Fig. \ref{f5-1-4}(a) shows the numerical results obtained with (\ref{bc-4}).
It can be seen that the maximum error converges at a rate of $O(J_{max}^{-1})$ for both cases
$m=2$ and $m=3$, much worse than the predicted order $O(J_{max}^{-2})$. 

To avoid this difficulty, we consider collocating the boundary condition
at the approximation points (cf. (\ref{interp-points})), i.e., 
\beq
u_0(\tilde{t}_{n,j}) = g(x_l(\tilde{t}_{n,j}), \tilde{t}_{n,j}),\quad u_{J_{max}}(\tilde{t}_{n,j}) = g(x_r(\tilde{t}_{n,j}), \tilde{t}_{n,j}),
\quad j = 1, ..., m.
\label{bc-5}
\eeq
The main advantage of this approximation is that the boundary condition is now enforced at $t = \tilde{t}_{n,m} = t_{n+1}$.
It is not difficult to show that (\ref{bc-4}) and (\ref{bc-5}) are equivalent for homogeneous Dirichlet boundary conditions but
different in general. Fig.  \ref{f5-1-4}(b) shows the results obtained with (\ref{bc-5}).
It can be seen that the second order convergence rate is recovered for all three cases.
\hfill \qed
\end{exam}

\begin{figure}[t!!!]
\centering
\hbox{
\begin{minipage}[t]{3.0in}
\centerline{(a)}
\includegraphics[width=3.0in]{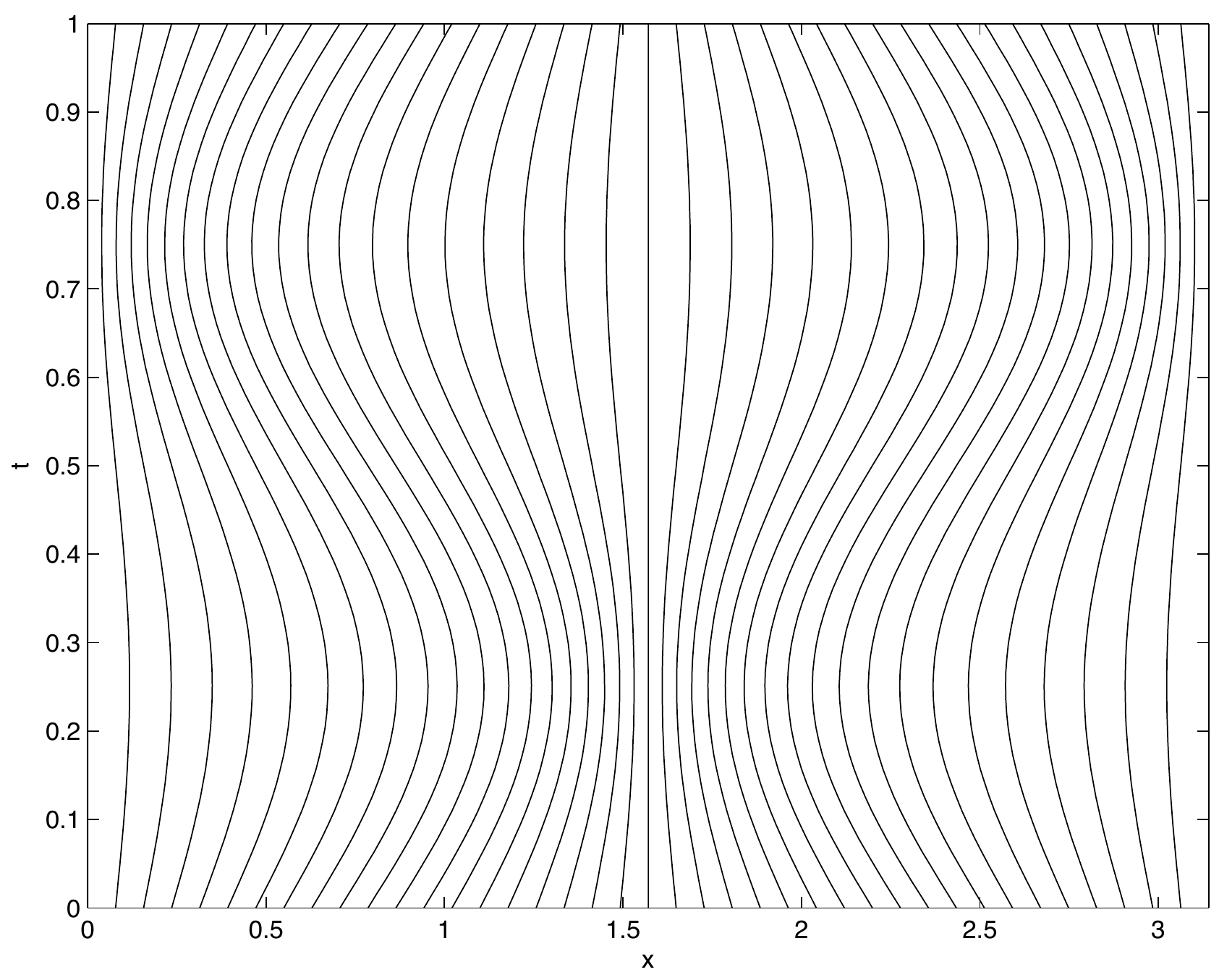}
\end{minipage}
\begin{minipage}[t]{3.0in}
\centerline{(b)}
\includegraphics[width=3.0in]{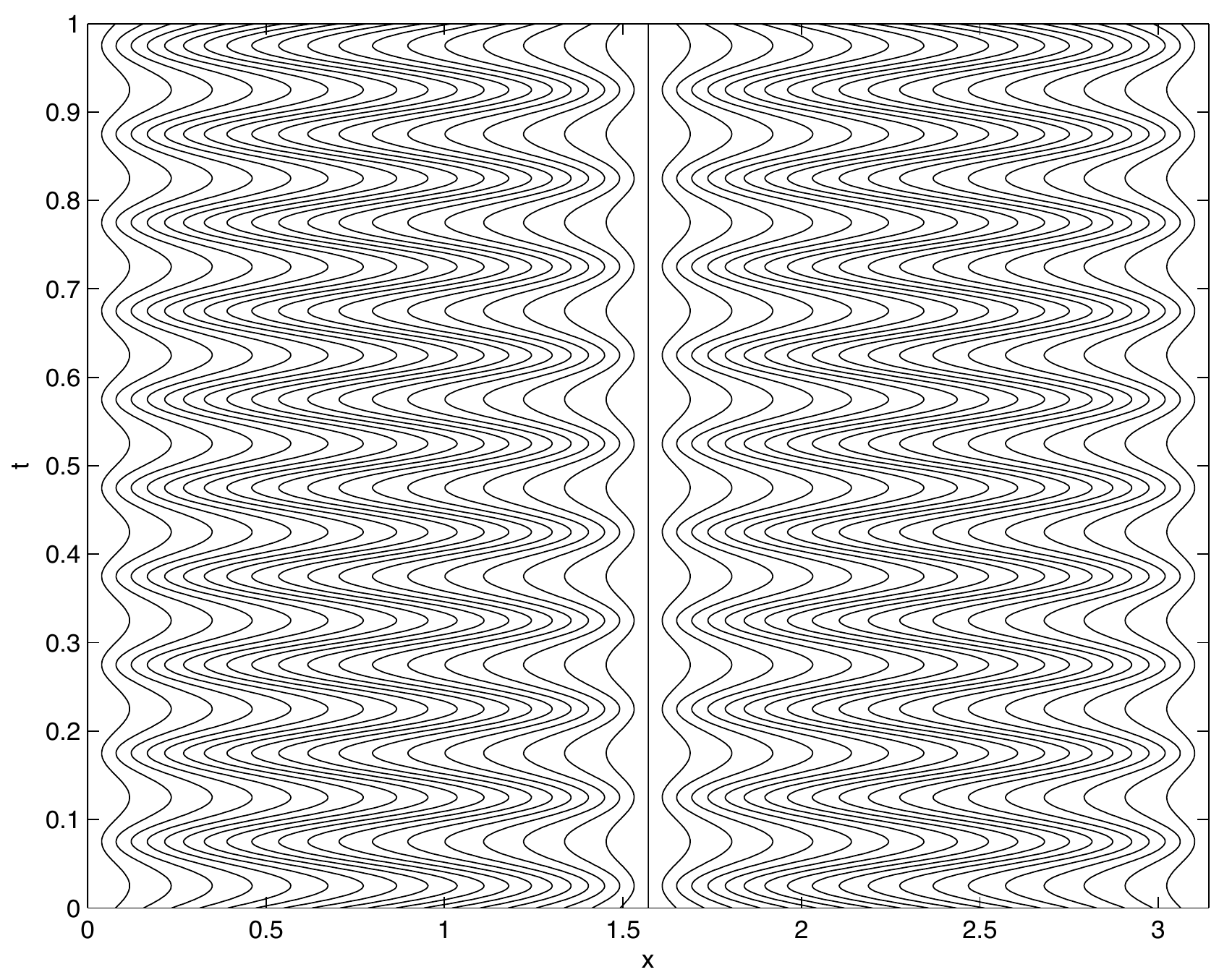}
\end{minipage}
}
\caption{Example~\ref{exam5.1}: The trajectories of two meshes of 41 points for $\omega = 2\pi$
and $\omega=20\pi$ are shown in (a) and (b), respectively.}
\label{f5-1-1}
\end{figure}

\begin{figure}[thb]
\centering
\hbox{
\begin{minipage}[t]{3.0in}
\centerline{(a) $m=1$}
\includegraphics[width=3.0in]{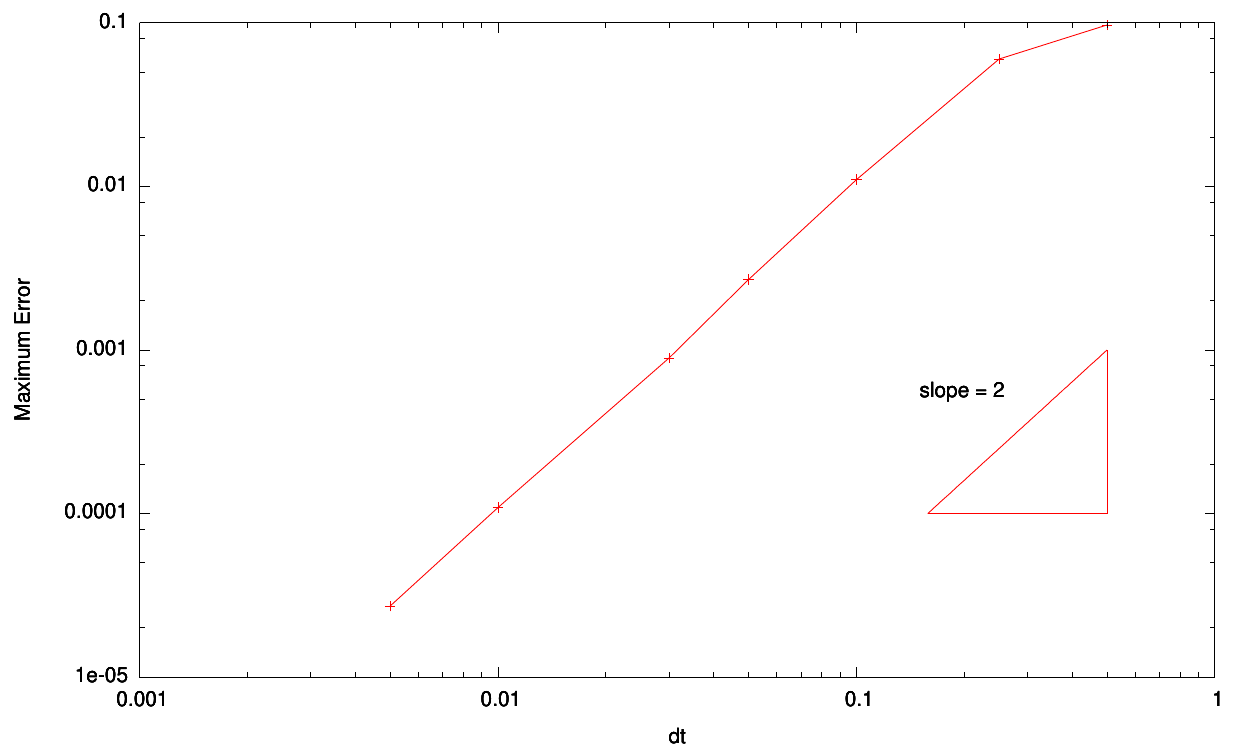}
\end{minipage}
\begin{minipage}[t]{3.0in}
\centerline{(b) $m=1,2,3$}
\includegraphics[width=3.0in]{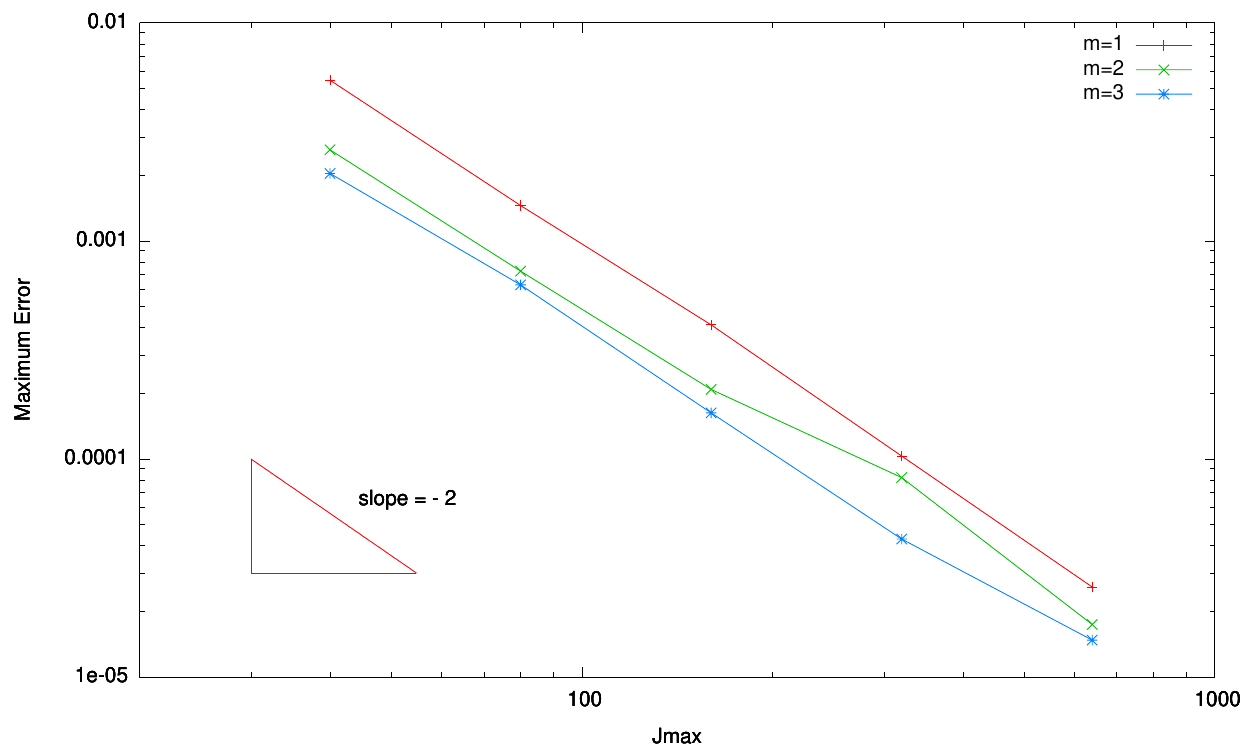}
\end{minipage}
}
\caption{The maximum solution error for scheme (\ref{col-1}) + (\ref{fd-2+1})
applied to Example \ref{exam5.1} ($\omega=2\pi$). The error is plotted in (a) as a function of $\Delta t$ for $J_{max}=1000$
and in (b) as a function of $J_{max}$ for $m=1$ ($\Delta t = \pi/J_{max}$), $m=2$ ($\Delta t = \sqrt{\pi/J_{max}}$),
and $m=3$ ($\Delta t = \sqrt[3]{\pi/J_{max}}$).}
\label{f5-1-2}
\end{figure}

\begin{figure}[thb]
\centering
\hbox{
\begin{minipage}[t]{3.0in}
\centerline{(a) $m=1$}
\includegraphics[width=3.0in]{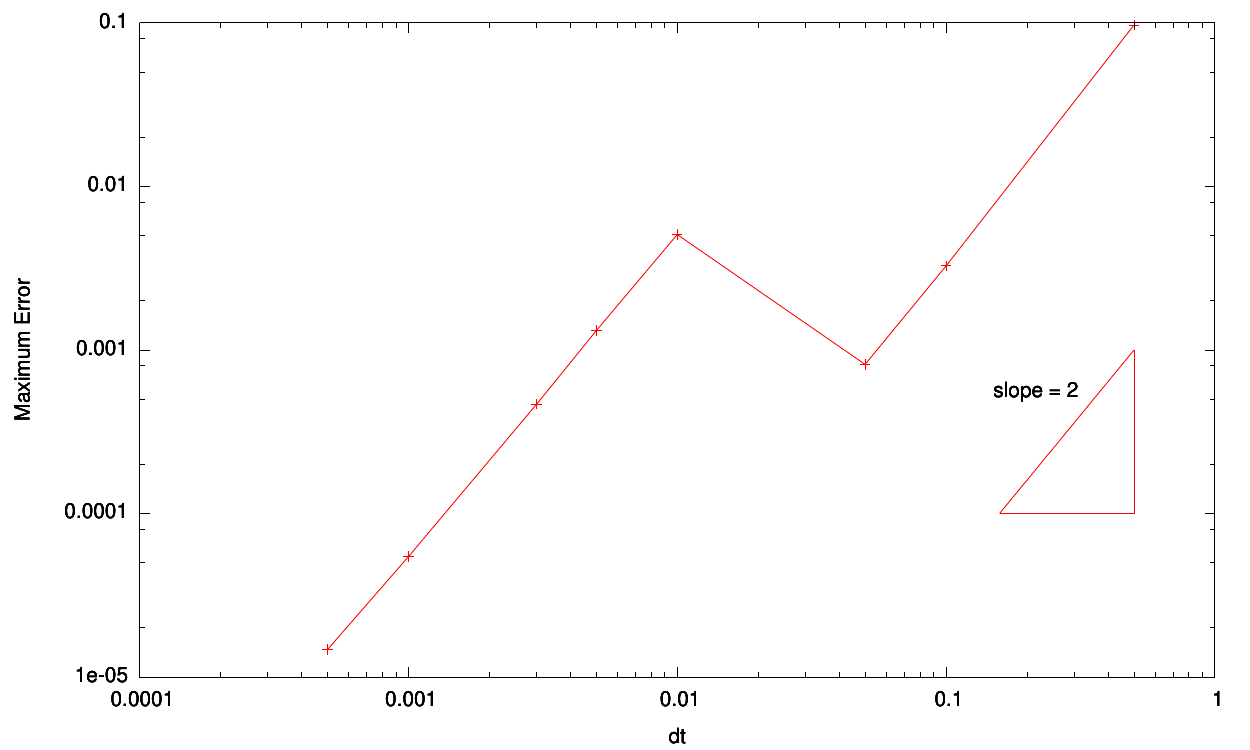}
\end{minipage}
\begin{minipage}[t]{3.0in}
\centerline{(b) $m=1$}
\includegraphics[width=3.0in]{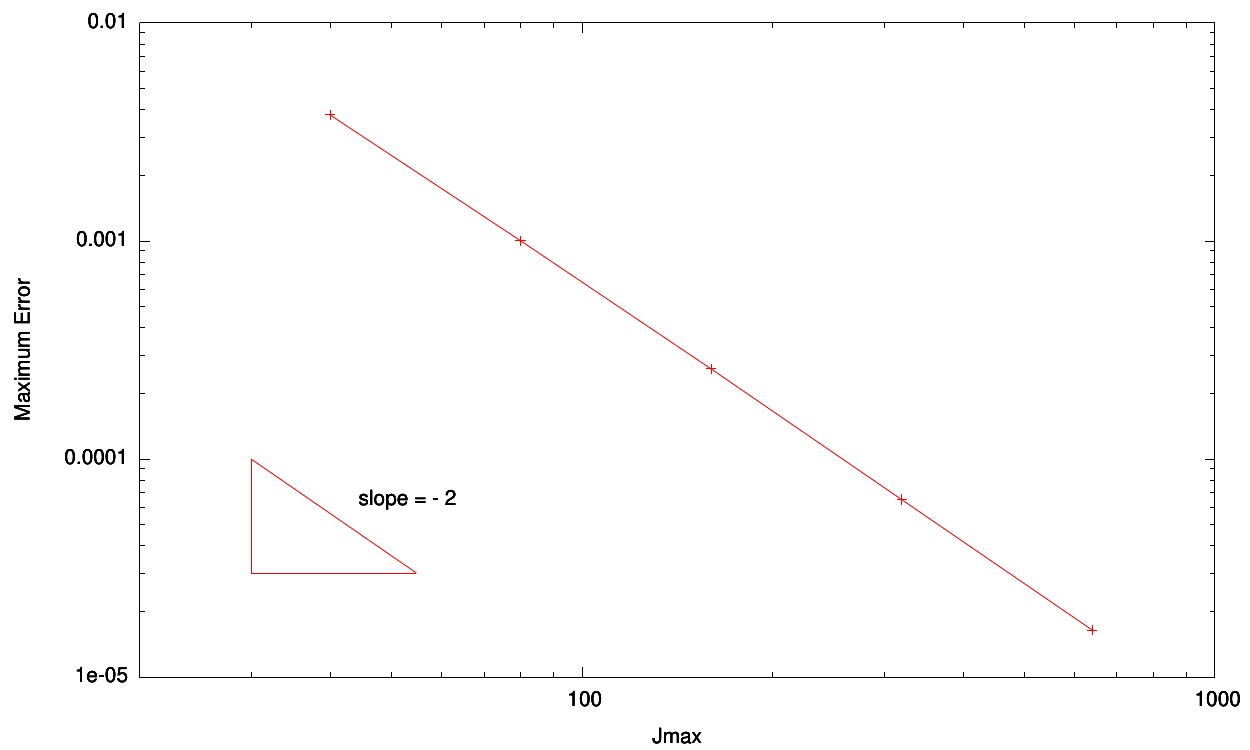}
\end{minipage}
}
\caption{The maximum error for scheme (\ref{col-1}) + (\ref{fd-2+1}) ($m=1$)
applied to Example \ref{exam5.1} ($\omega=20 \pi$). The error is plotted in (a) as a function of $\Delta t$ for $J_{max}=1000$
and in (b) as a function of $J_{max}$ ($\Delta t = 0.1 \pi/J_{max}$).}
\label{f5-1-3}
\end{figure}

\begin{figure}[thb]
\centering
\hbox{
\begin{minipage}[t]{3.0in}
\centerline{(a) BC approximation (\ref{bc-4})}
\includegraphics[width=3.0in]{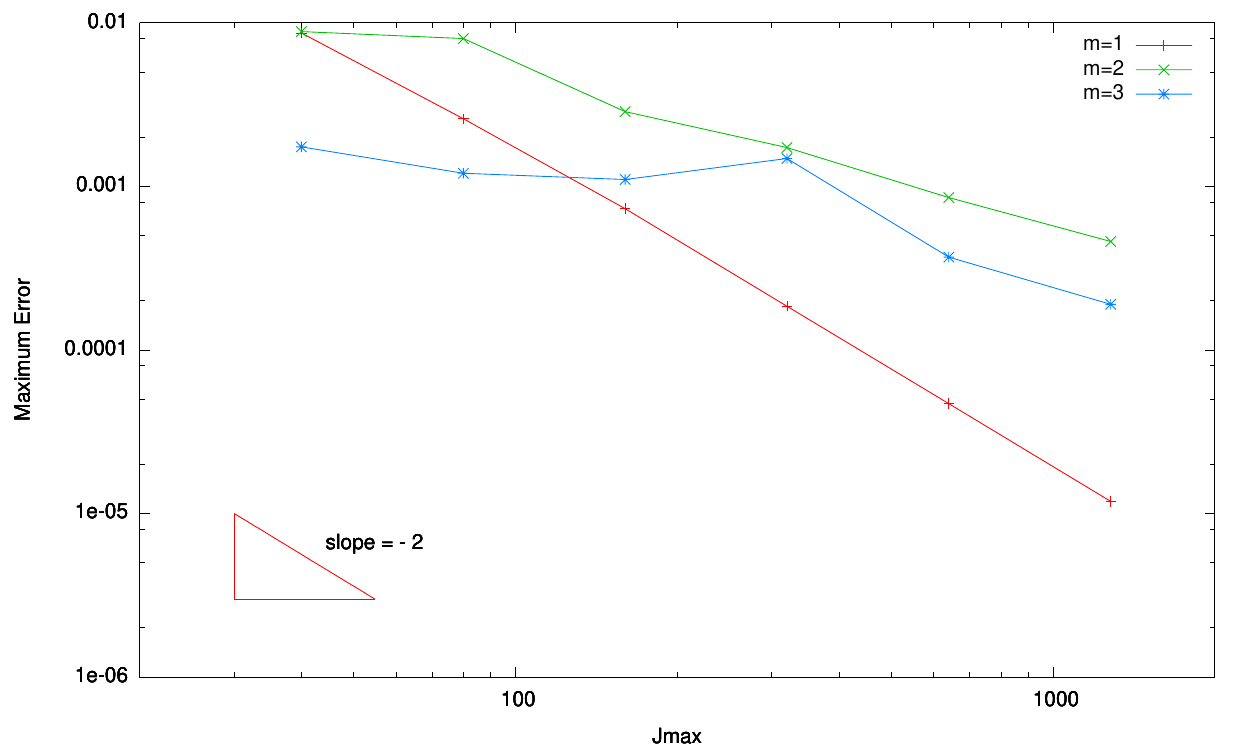}
\end{minipage}
\begin{minipage}[t]{3.0in}
\centerline{(b) BC approximation (\ref{bc-5})}
\includegraphics[width=3.0in]{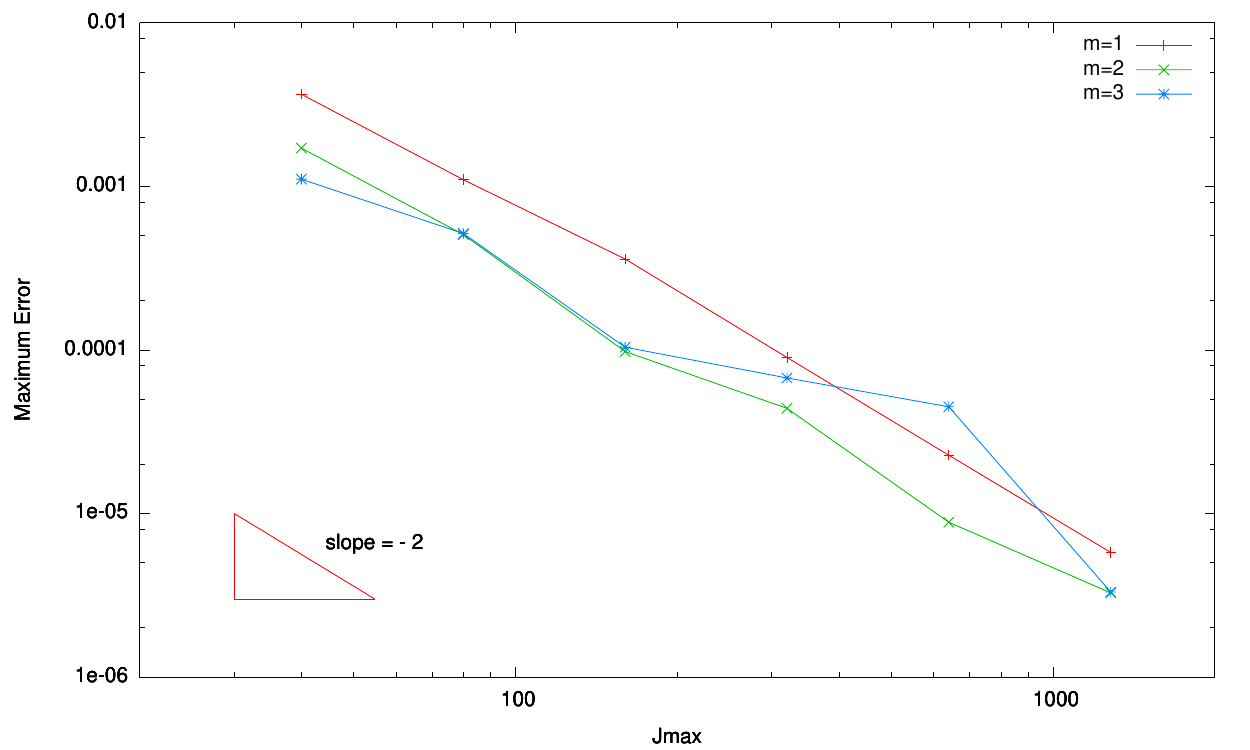}
\end{minipage}
}
\caption{The maximum error for scheme (\ref{col-1}) + (\ref{fd-2+1}) ($m=1,2,3$)
applied to Example \ref{exam5.1} ($\omega=2 \pi$) with exact solution (\ref{exam5.1-4}).
The error is plotted as a function of $J_{max}$ with
$\Delta t = (\pi/J_{max})^{1/m}$.}
\label{f5-1-4}
\end{figure}

\begin{exam}
\label{exam5.2}
We consider the one dimensional  heat equation with a moving domain. The problem is 
in the form (\ref{ibvp-1d}) with
\beq
\begin{cases}
a(x,t) = 1,\quad b(x,t) = 0,\quad c(x,t) = 0,\\
x_l(t) = \frac{\pi}{3} \sin(\omega t),\quad x_r(t) = \pi - \frac{\pi}{3} \sin(\omega t),
\end{cases}
\label{exam5.2-1}
\eeq
where the parameter $\omega$ is used to control the speed of the boundary movement.
The source term $f(x,t)$, initial solution, and Dirichlet boundary condition are chosen
such that the exact solution of the IBVP is given by
\beq
u_{exact}(x,t) = \sin\left (\frac{\pi (x- x_l(t))}{x_r(t)-x_l(t)}\right ) (2 + \sin(\pi t)) .
\label{exam5.2-2}
\eeq
The mesh is defined as
\beq
x_j = x_l(t) + \frac{j}{J_{max}} (x_r(t)-x_l(t)),\quad j = 0, ..., J_{max}.
\label{exam5.2-2}
\eeq
The trajectories of two meshes with $\omega = 2\pi$ and $20 \pi$ are shown in Fig. \ref{f5-2-1} 

Recall that the mesh is treated linearly on each time interval $[t_n, t_{n+1}]$. As a consequence,
when the boundary of the domain is moving, the first and last mesh points, $x_0(t)$ and $x_{J_{max}}(t)$,
generally do not coincide with the boundary points $x_l(t)$ and $x_r(t)$ for $t \in (t_n, t_{n+1})$;
see the illustration in Fig. \ref{f5-2-2}.
In this situation, a more accurate approximation of the boundary conditions than (\ref{bc-4}) or (\ref{bc-5}) is needed.
By expanding $u(x_0(t), t)$ and $u(x_1(t), t)$ about $x=x_l(t)$ and $u(x_{J_{max}}(t), t)$ and $u(x_{J_{max}-1}(t), t)$
about $x_r(t)$, we get
\beq
\begin{cases}
(x_1-x_l) (u_0 - g(x_l, t)) - (x_0-x_l) (u_1-g(x_l, t))  = 0,\\
(x_{J_{max}-1}-x_r) (u_{J_{max}}-g(x_r, t)) - (x_{J_{max}}-x_r) (u_{J_{max}-1}-g(x_r, t))  = 0.
\end{cases}
\label{bc-6}
\eeq
As discussed in Example \ref{exam5.1}, the above conditions are imposed at the approximation points,
$\tilde{t}_{n,1},...,\tilde{t}_{n,m}$ (cf. (\ref{interp-points})).

Numerical results obtained with scheme (\ref{col-1}) + (\ref{fd-2+1}) ($m=1,2,3$)
for $\omega = 2\pi$ and $20\pi$ are shown
in Figs. \ref{f5-2-3} and \ref{f5-2-4}, respectively. It can be seen that the scheme is stable (the solution is bounded)
and the error is of order $O(\Delta t^{2m}) + O(\Delta x^2)$ ($m=1,2,3$), consistent with the
theoretical prediction.
\hfill \qed
\end{exam}

\begin{figure}[thb]
\centering
\hbox{
\begin{minipage}[t]{3.0in}
\includegraphics[width=3.0in]{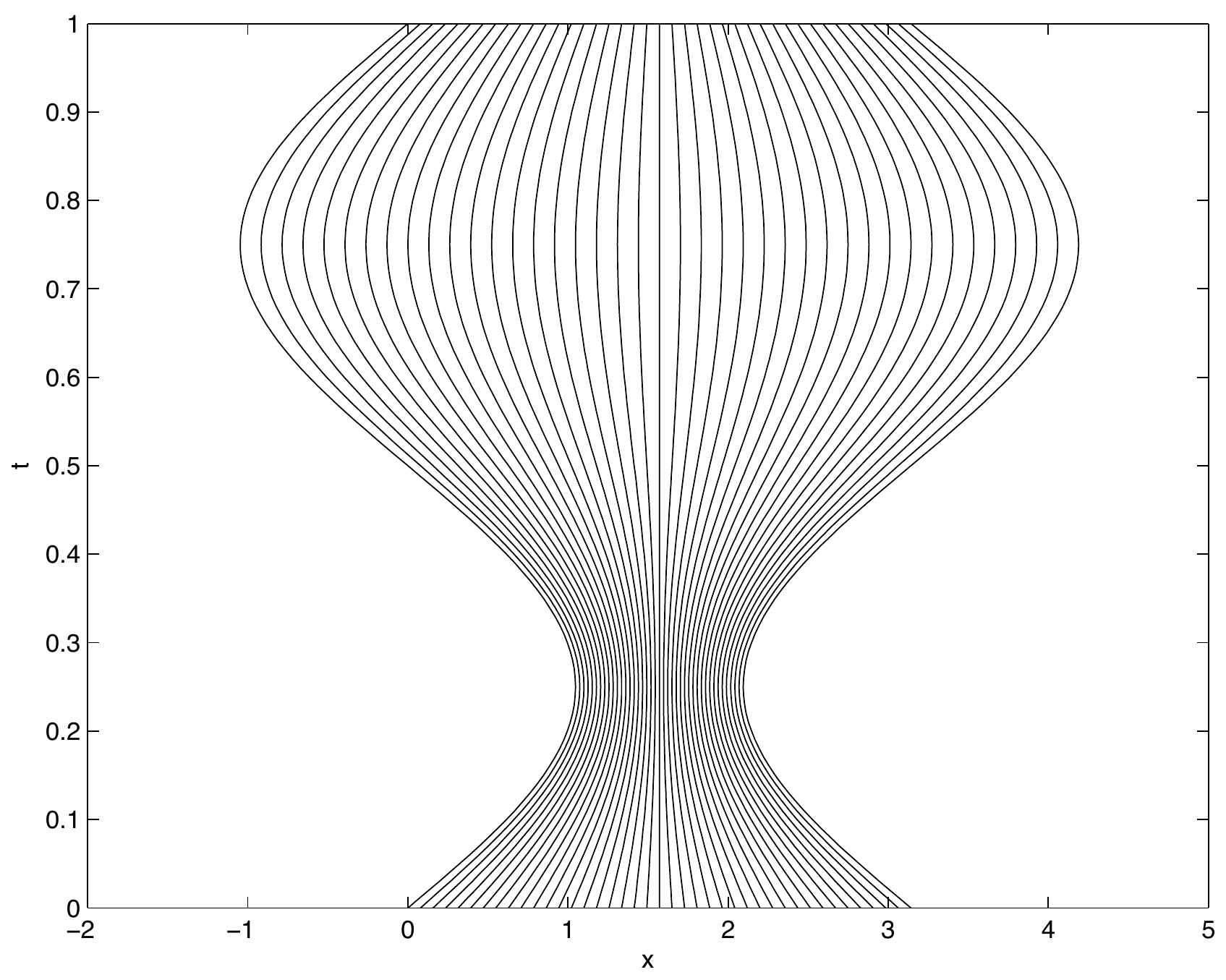}
\end{minipage}
\begin{minipage}[t]{3.0in}
\includegraphics[width=3.0in]{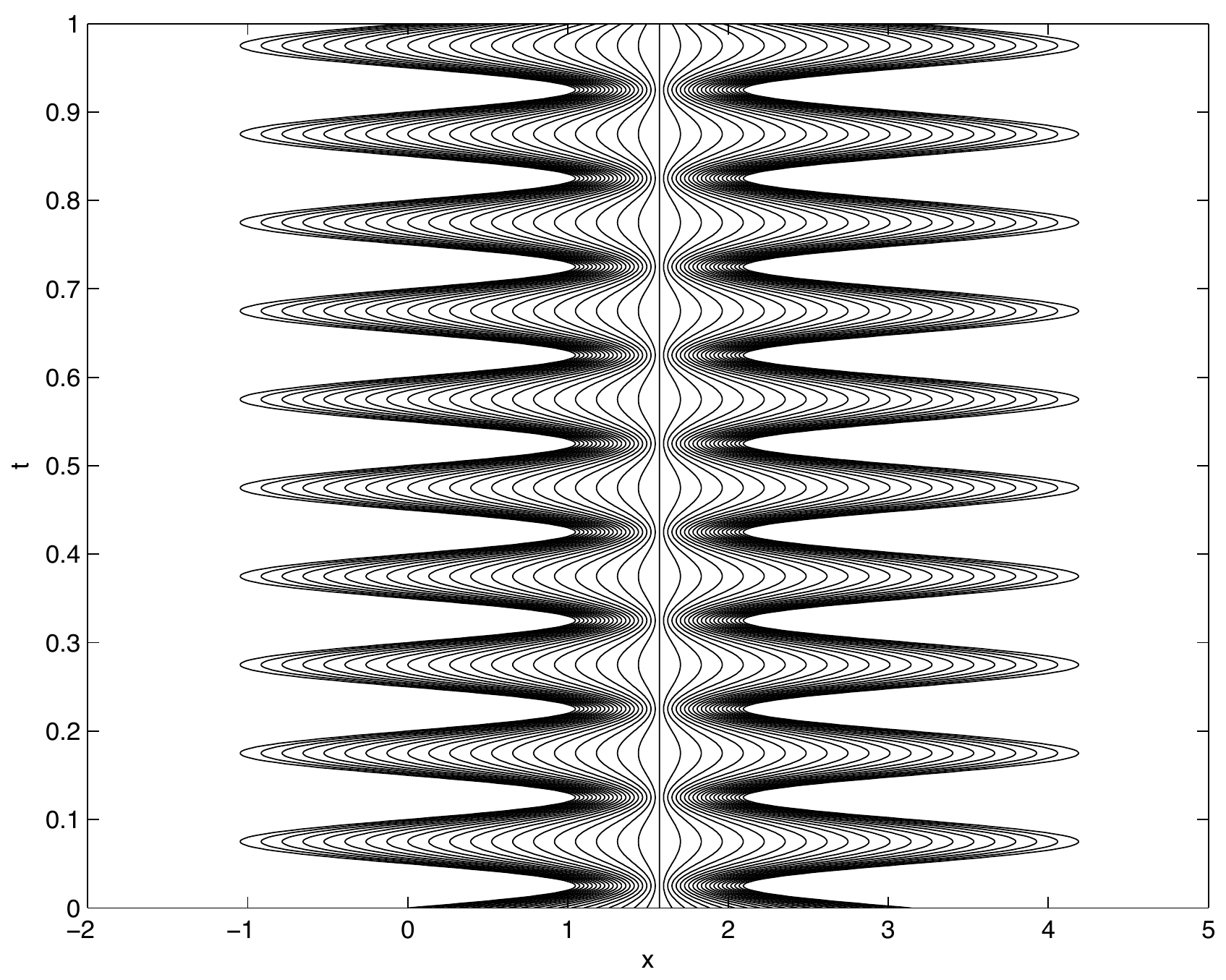}
\end{minipage}
}
\caption{Mesh trajectories of two meshes of 41 points for $\omega = 2\pi$ and $20\pi$, respectively.}
\label{f5-2-1}
\end{figure}

\begin{figure}[thb]
\centering
\begin{tikzpicture}[scale = 1.5 ]
\draw[thick] (-0.5,-0.5) parabola (1.5,3) node[above] {Left boundary $x_l(t)$};
\draw [-] (-1,0) -- (4,0) node[right]{$t_n$};
\draw [-] (-1,2) -- (4,2) node[right]{$t_{n+1}$};
\filldraw[black] (0.25,0) circle(1pt) node[below right] {$x_0^n$};
\filldraw[black] (1.19,2) circle(1pt);
\draw[above right] (1.25, 2) node {$x_0^{n+1}$};
\draw[-] (0.25,0) -- (1.19,2);
\draw [-] (-1,1) -- (4,1) node[right]{$t$};
\filldraw[black] (0.72,1) circle(1pt);
\draw [->] (0.47,1.25) -- (0.62,1.1);
\draw[above] (0.47,1.25) node {$x_0(t)$};
\filldraw[black] (0.81,1) circle(1pt);
\draw [->] (1.16,1.35) -- (0.91,1.1);
\draw[above right] (1.16,1.35) node {$x_l(t)$};
\end{tikzpicture}
\caption{A sketch of the boundary $x=x_l(t)$ and the first mesh point $x=x_0(t)$. Generally speaking,
$x_0(t)$ does not coincide with $x_l(t)$ for $t \in (t_n,t_{n+1})$.}
\label{f5-2-2}
\end{figure}
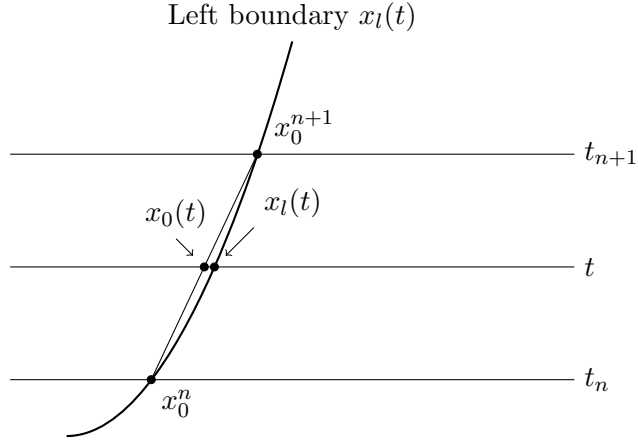

\begin{figure}[thb]
\centering
\hbox{
\begin{minipage}[t]{3.0in}
\centerline{(a) $m=1$}
\includegraphics[width=3.0in]{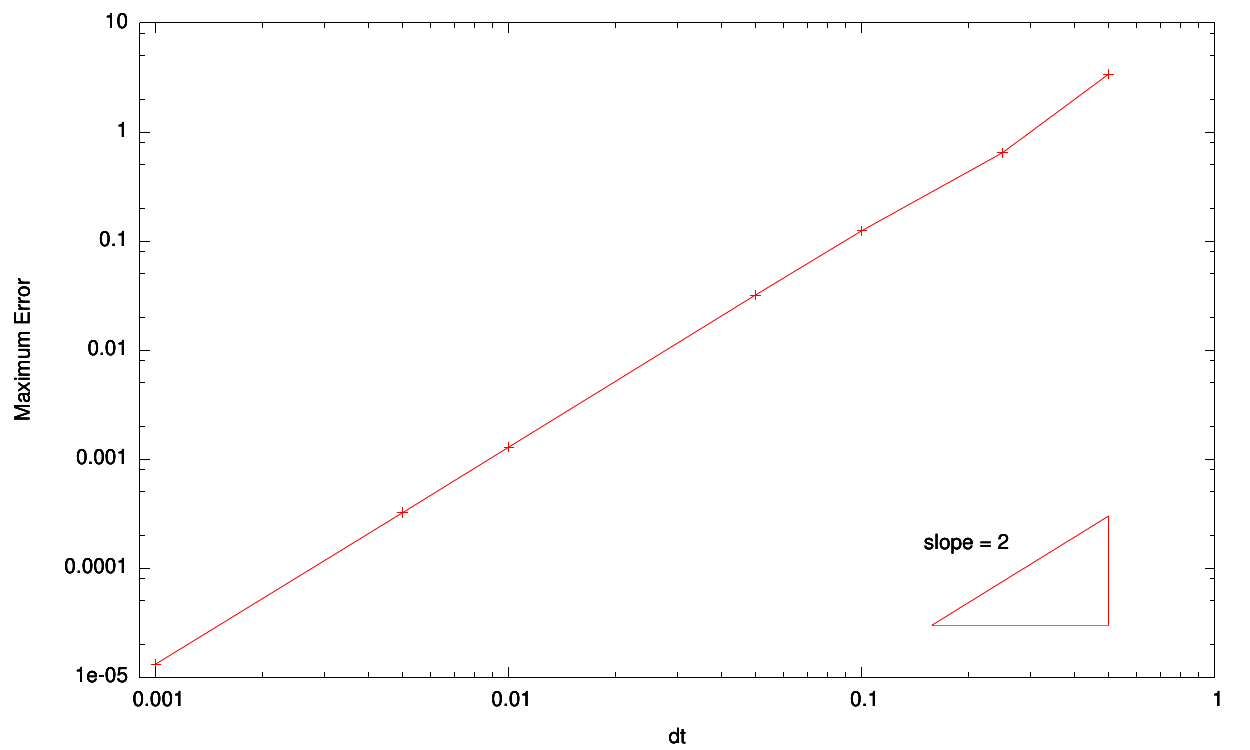}
\end{minipage}
\begin{minipage}[t]{3.0in}
\centerline{(b) $m=1,2,3$}
\includegraphics[width=3.0in]{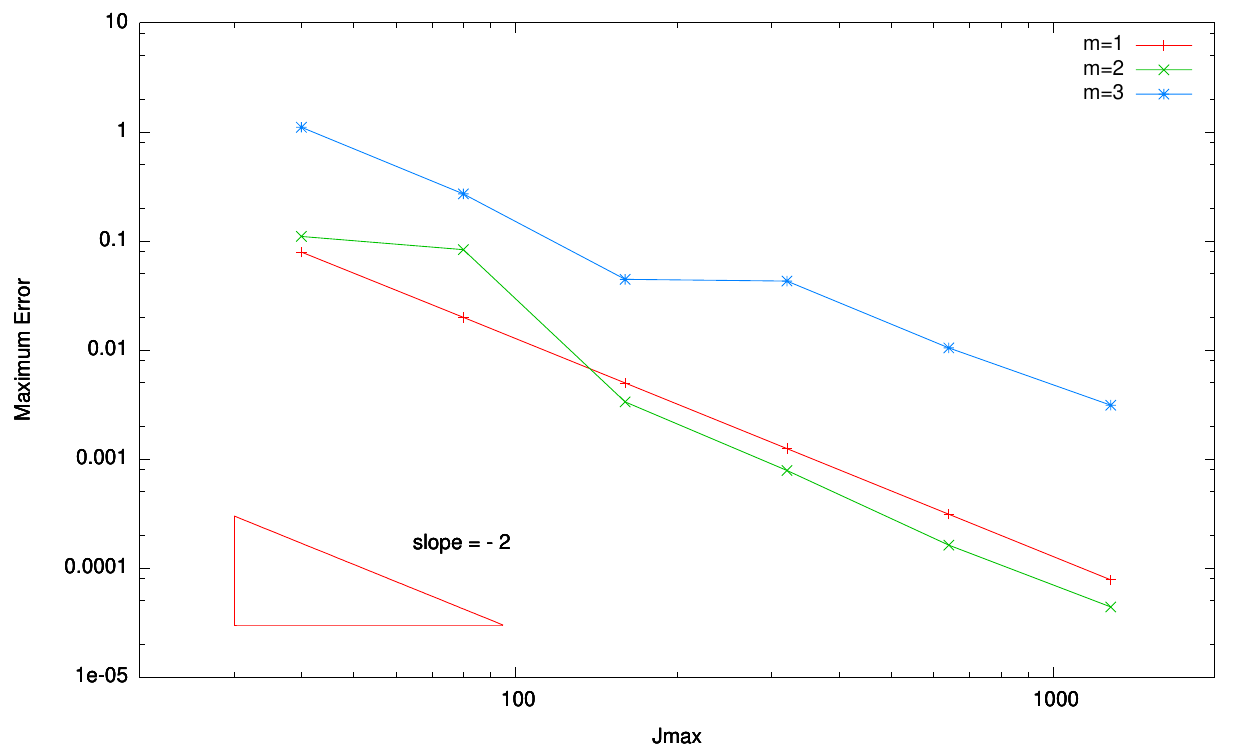}
\end{minipage}
}
\caption{The maximum error for scheme (\ref{col-1}) + (\ref{fd-2+1}) ($m=1,2,3$)
applied to Example \ref{exam5.2} ($\omega=2\pi$). The error is plotted in (a) as a function of $\Delta t$ for $J_{max}=1000$
and in (b) as a function of $J_{max}$ for $\Delta t = (\pi/J_{max})^{1/m},\; m =1,2,3$.}
\label{f5-2-3}
\end{figure}

\begin{figure}[thb]
\centering
\hbox{
\begin{minipage}[t]{3.0in}
\centerline{(a): $m=1$}
\includegraphics[width=3.0in]{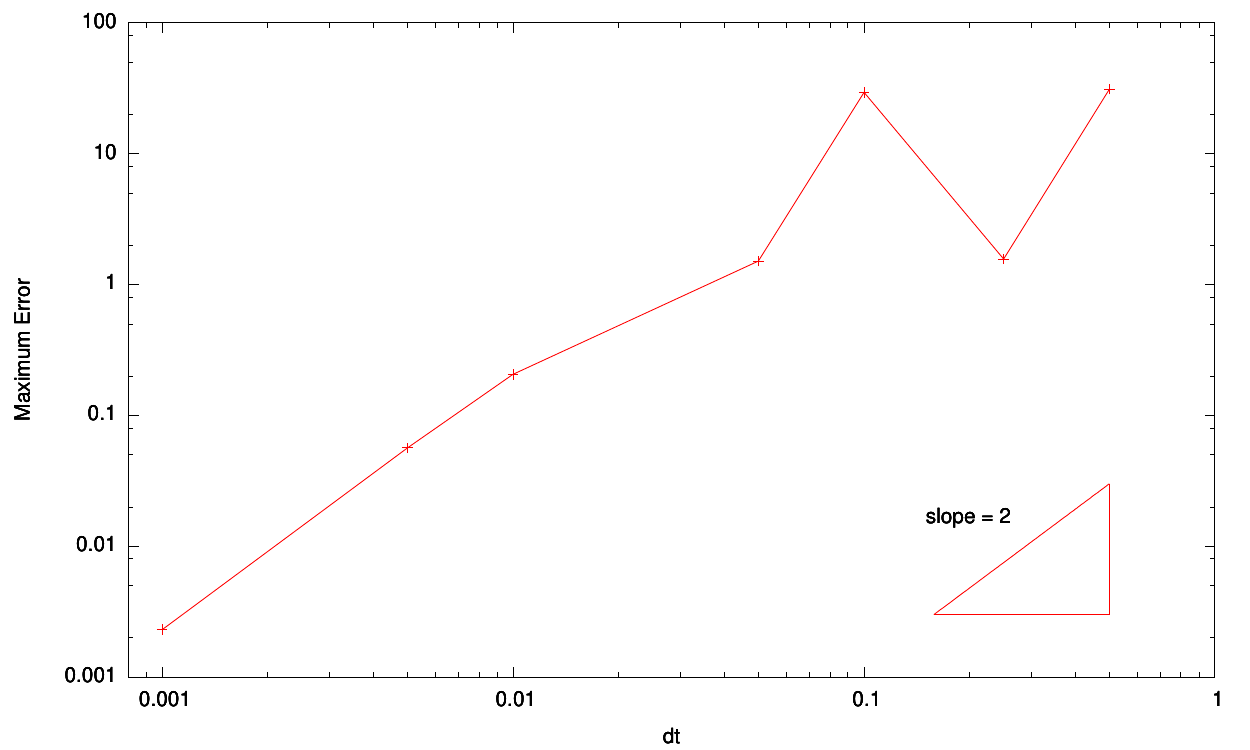}
\end{minipage}
\begin{minipage}[t]{3.0in}
\centerline{(b): $m=1$}
\includegraphics[width=3.0in]{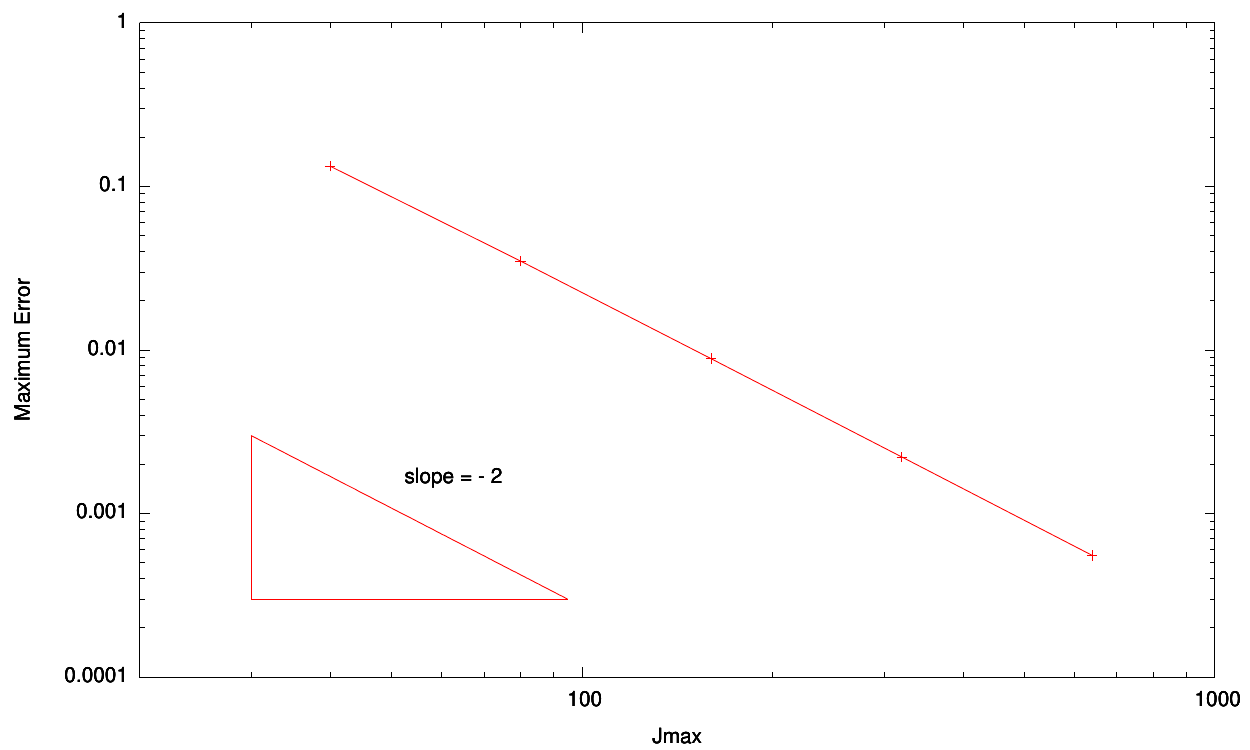}
\end{minipage}
}
\caption{The maximum error for scheme (\ref{col-1}) + (\ref{fd-2+1}) (with $m=1$)
applied to Example \ref{exam5.2} ($\omega=20\pi$). The error is plotted in (a) as a function of $\Delta t$ for $J_{max}=1000$
and in (b) as a function of $J_{max}$ for $\Delta t = 0.1\pi/J_{max}$.}
\label{f5-2-4}
\end{figure}

\begin{exam}
\label{exam5.3}
The last example is a two dimensional problem in the form of IBVP (\ref{ibvp}) with
\beq
\begin{cases}
a(x,y,t) = 1,\quad b_1(x,y,t) = b_2(x,y,t)=0,\quad c(x,t) = 0,\\
\Omega = (0,\pi)\times (0,\pi) .
\end{cases}
\label{exam5.3-1}
\eeq
The source term $f(x,y,t)$, initial solution, and Dirichlet boundary condition are chosen
such that the exact solution of the IBVP is given by
\beq
u_{exact}(x,y,t) = (2 + \sin(\pi t)) \sin(x) \sin(y) .
\label{exam5.3-2}
\eeq
The moving mesh is generated using the coordinate transformation 
\beq
x = \xi + 0.2 \sin(2 \xi) \sin( 2 \eta) \sin( \omega t),\quad
y = \eta + 0.2 \sin(2 \xi) \sin( 2 \eta) \sin( \omega t),\qquad (\xi, \eta) \in \Omega
\label{exam5.3-3}
\eeq
where the parameter $\omega$ is used to control the speed of mesh movement and
a rectangular mesh of $(J+1)(K+1)$ points is used for the $\xi$-$\eta$ domain. In our computation $\omega = 2\pi$ is used.
A $41\times 41$ moving mesh is shown in Fig. \ref{f5-3-1} for $t = 0.25$ and $t=0.75$.
The numerical results are shown in Fig.~\ref{f5-3-2}, which are consistent with the observations made for
the 1D examples and with the theoretical prediction.
\hfill \qed
\end{exam}

\begin{figure}[thb]
\centering
\hbox{
\begin{minipage}[t]{3.0in}
\centerline{(a) $t=0.25$}
\includegraphics[width=3.0in]{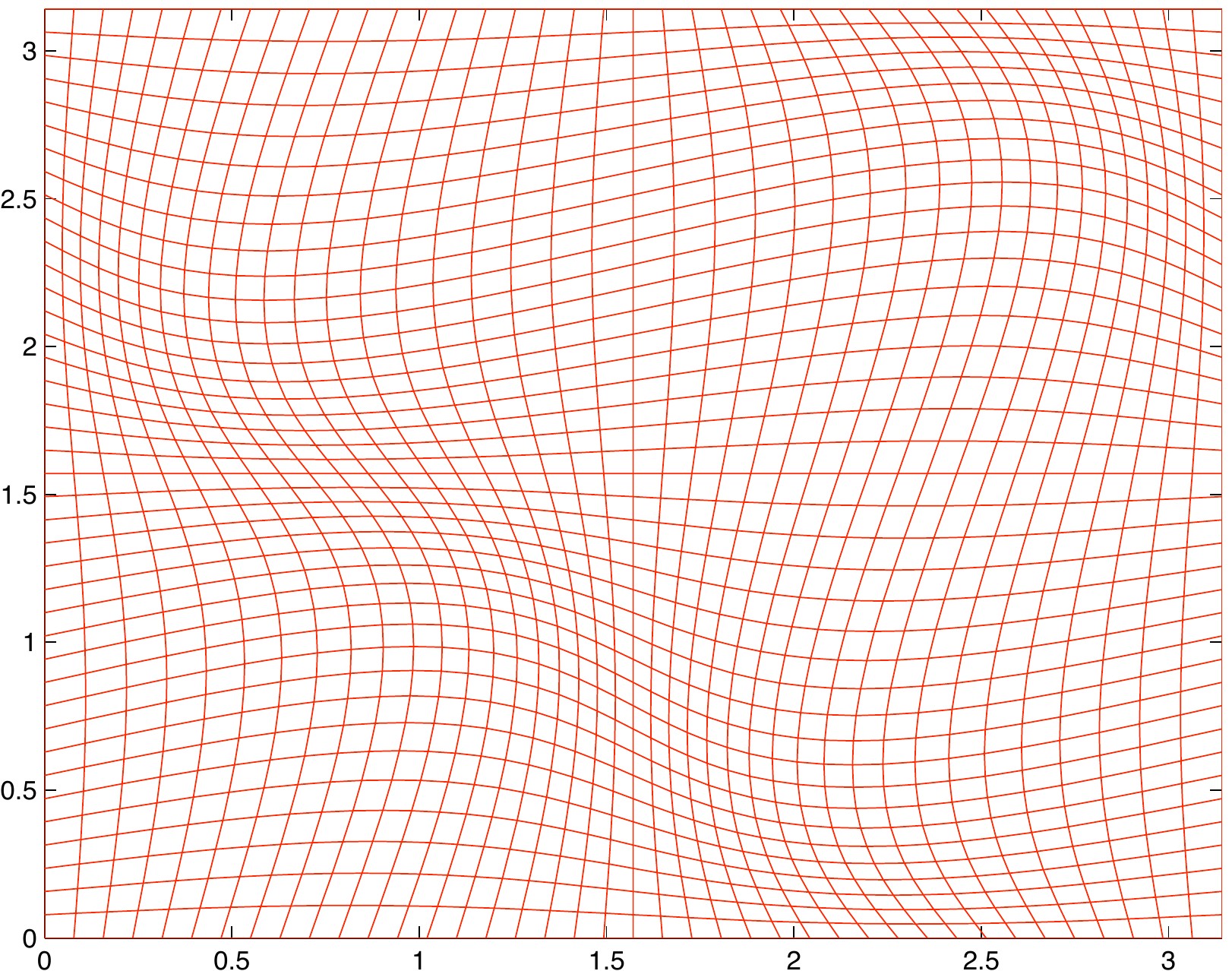}
\end{minipage}
\hspace{5mm}
\begin{minipage}[t]{3.0in}
\centerline{(b) $t=0.75$}
\includegraphics[width=3.0in]{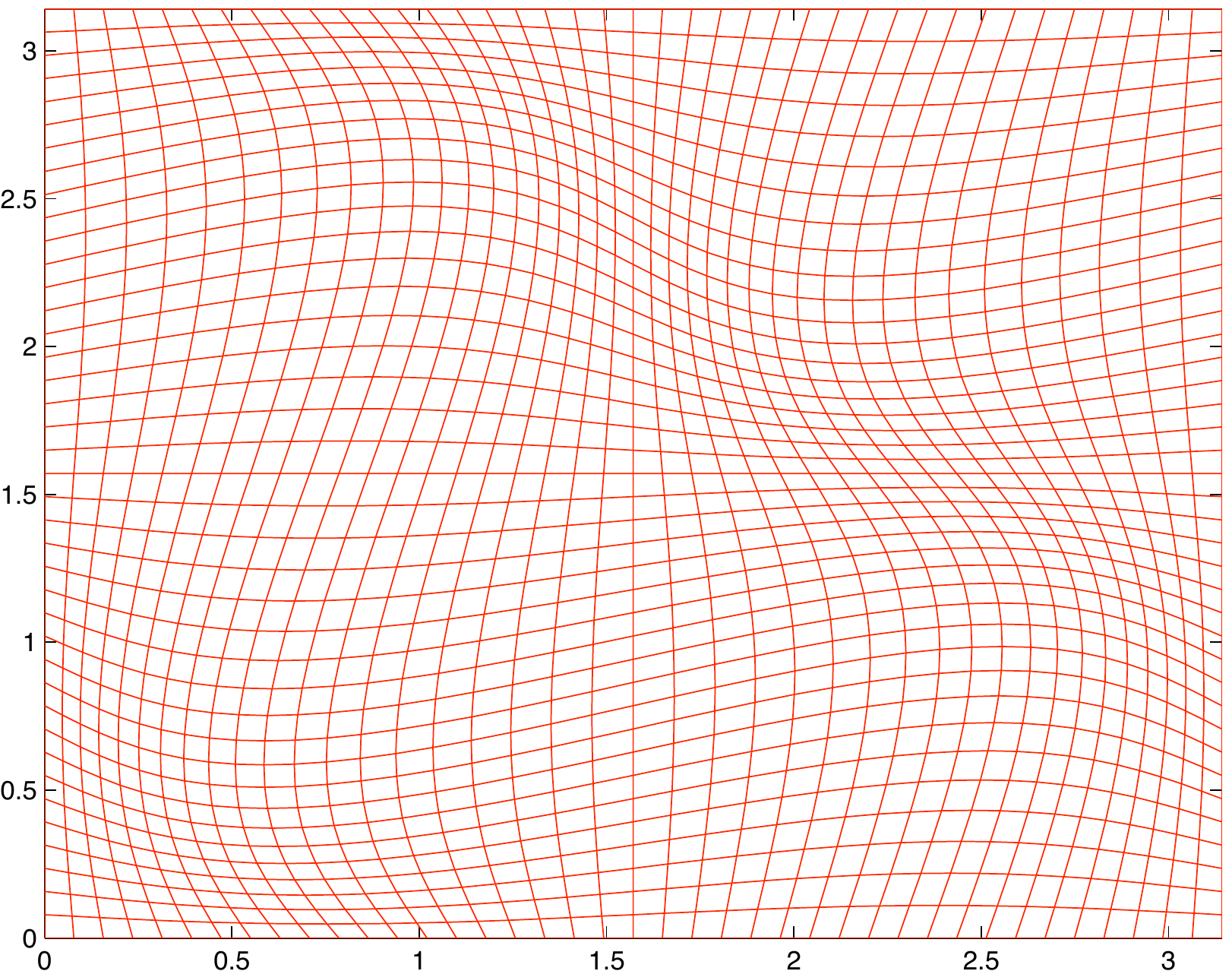}
\end{minipage}
}
\caption{A $41\times 41$ moving mesh for Example \ref{exam5.3} is shown for $t = 0.25$ and $t=0.75$.}
\label{f5-3-1}
\end{figure}

\begin{figure}[thb]
\centering
\hbox{
\begin{minipage}[t]{3.0in}
\centerline{(a) $m=1$}
\includegraphics[width=3.0in]{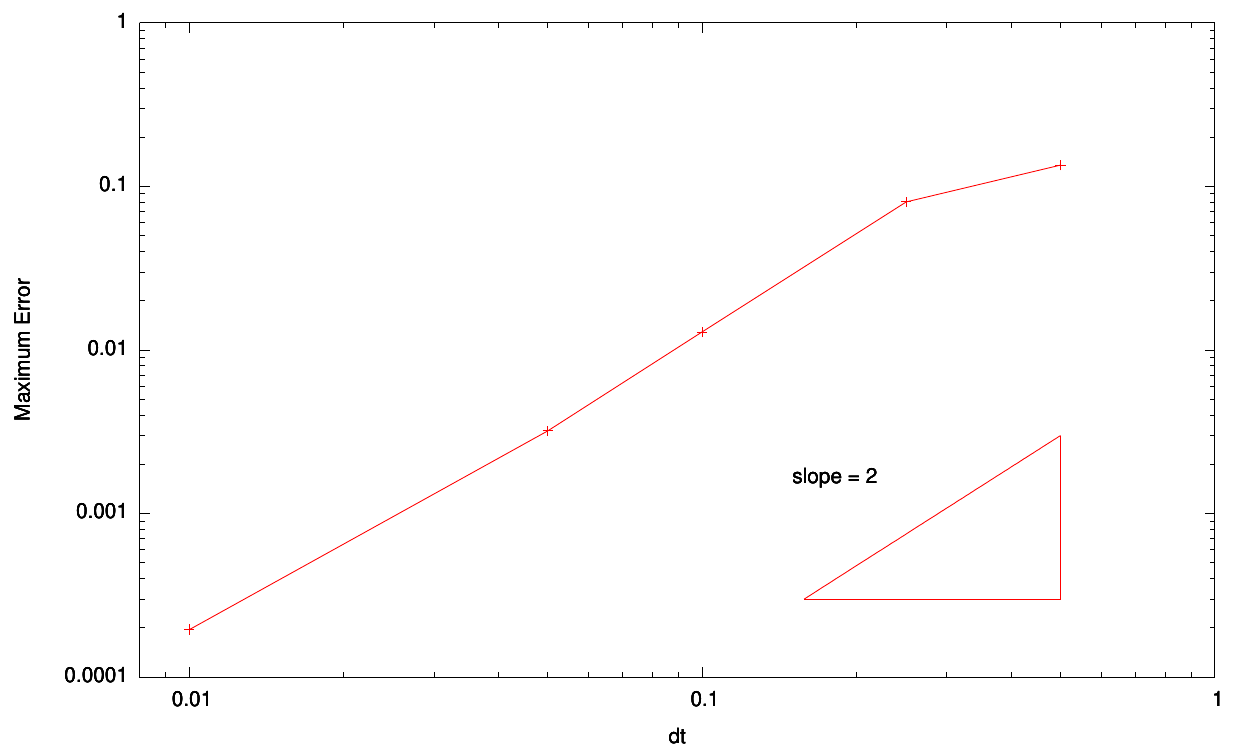}
\end{minipage}
\begin{minipage}[t]{3.0in}
\centerline{(b) $m=1,2,3$}
\includegraphics[width=3.0in]{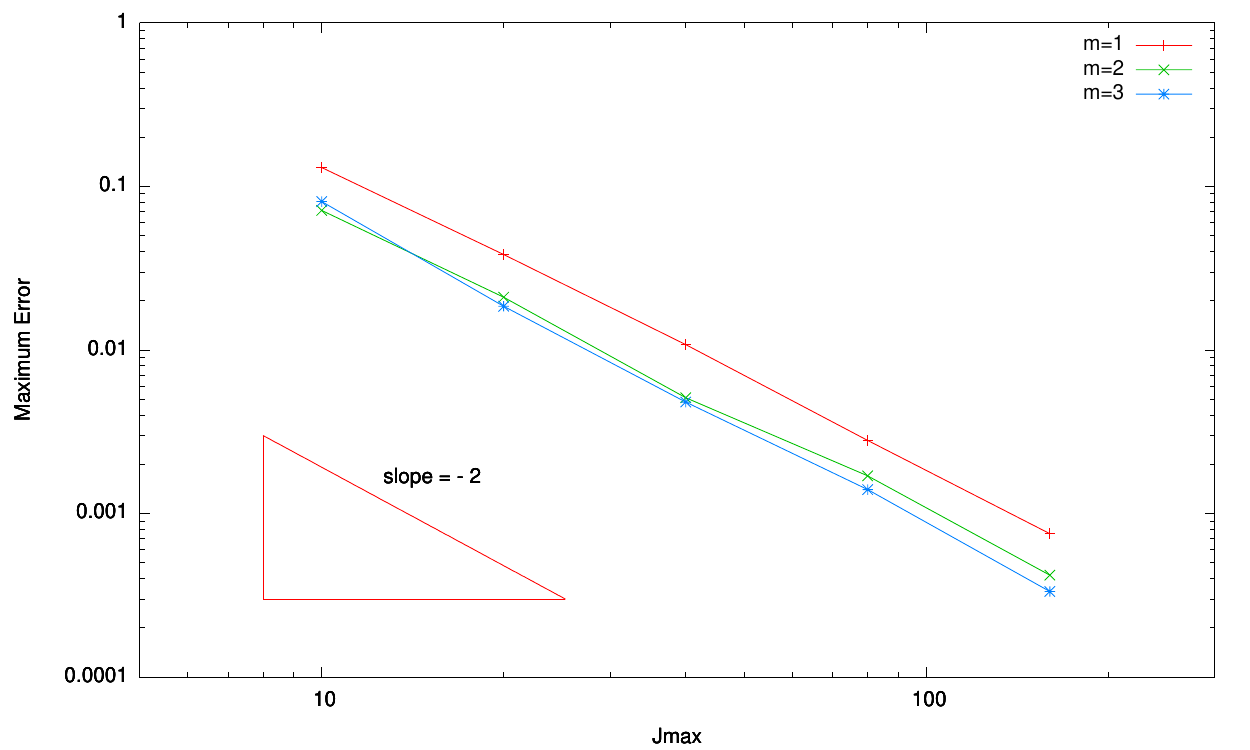}
\end{minipage}
}
\caption{The maximum error for scheme (\ref{col-1}) applied to (\ref{ode-3}) with (\ref{fd-2d-3}) and (\ref{fd-2d-4})
for Example \ref{exam5.3} ($\omega=2\pi$). The error is plotted in (a) as a function of $\Delta t$ for
$J_{max}= K_{max}=160$ ($J_{max}= K_{max}=320$ was used for $\Delta t = 0.01$)
and in (b) as a function of $J_{max}$ for $\Delta t = (\pi/J_{max})^{1/m},\; m =1,2,3$.}
\label{f5-3-2}
\end{figure}

\section{Conclusions and comments}

In the previous sections we have developed a family of finite difference schemes for linear convection-diffusion
equations on moving meshes. Those schemes can be of second and higher order in time, preserve a stability
inequality, and are unconditionally stable in the sense that they impose no constraint on time step size for stability purpose.
More specifically, scheme (\ref{col-1}) has been developed 
for ODE system of the form (\ref{ode-3}) with property (\ref{cond-1})
by first transforming (\ref{ode-3}) into (\ref{ode-v-1})
and then applying the $m$-point collocation scheme to the transformed system.
Several finite difference discretizations for one dimensional  and two dimensional convection-diffusion equations
on moving meshes have been constructed and shown to satisfy property (\ref{cond-1}) (cf. Theorems~\ref{thm3.1} and
\ref{thm4.1}). Numerical results presented in Section \ref{SEC:numerics} verify the theoretically predicted
stability and convergence order of those schemes.

Several generalizations of the current work are under investigation. It is interesting to know how
the current strategy can be used for nonlinear differential equations and wave equations and
how it can be combined with the method of lines for general differential equations. Moreover,
it is noted that scheme (\ref{col-1}) may not work efficiently
when mass matrix $M(t)$ is not diagonal (as in the case with finite element approximation).
Development of an efficient implementation in the finite element case certainly deserves further investigations.


\end{document}